\documentclass{amsart}

\usepackage{graphicx}
\usepackage[centertags]{amsmath}
\usepackage{amsfonts}
\usepackage{amssymb}
\usepackage{amsthm}
\usepackage{newlfont}
\usepackage{hyperref}
\usepackage{mathrsfs} 
\usepackage{yfonts} 

\newtheorem*{thm*}{Theorem}
\newtheorem{thm}{Theorem}[section]
\newtheorem{cor}[thm]{Corollary}
\newtheorem{lem}[thm]{Lemma}
\newtheorem{prop}[thm]{Proposition}

\theoremstyle{definition}
\newtheorem{defn}[thm]{Definition}

\theoremstyle{remark}

\newtheorem{example}[thm]{Example}
\numberwithin{equation}{section}

\DeclareSymbolFont{bbold}{U}{bbold}{m}{n}
\DeclareSymbolFontAlphabet{\mathbbold}{bbold}

\newcommand{\N}{\mathbb{N}}   
\newcommand{\Z}{\mathbb{Z}}   
\newcommand{\R}{\mathbb{R}}   
\renewcommand{\:}{\,:\,}      

\newcommand{\symd}{\triangle}    
\newcommand{\res}{\restriction}  

\newcommand{\acts}{\curvearrowright}             
\newcommand{\Stab}{\mathrm{Stab}}                
\newcommand{\Sym}{\mathrm{Sym}}                  
\newcommand{\Aut}{\mathrm{Aut}}                  
\newcommand{\pmp}{p{$.$}m{$.$}p{$.$}}            
\newcommand{\Borel}{\mathcal{B}}                 
\newcommand{\malg}{\mathrm{MALG}}                
\newcommand{\Null}{\mathfrak{N}}                 
\newcommand{\sinv}{\mathscr{I}}                  
\newcommand{\sfinv}{\mathscr{I}^{\text{fin}}}    

\newcommand{\E}{\mathscr{E}}   
\newcommand{\M}{\mathscr{M}}   
\newcommand{\Act}{A}           
\newcommand{\PAct}{A^{\text{aper}}}
\newcommand{\EAct}{A^{\text{erg}}}

\newcommand{\pv}{\bar{p}} 

\newcommand{\cL}{\mathscr{L}} 
\newcommand{\cP}{\mathcal{P}} 
\newcommand{\cQ}{\mathcal{Q}} 

\newcommand{\cF}{\mathcal{F}} 

\newcommand{\Prt}{\mathscr{P}}      
\newcommand{\HPrt}{\mathscr{P}}     
\newcommand{\dB}{d}                 
\newcommand{\dR}{d^{\mathrm{Rok}}}  
\newcommand{\dist}{\mathrm{dist}}   
\newcommand{\sH}{\mathrm{H}}        
\newcommand{\given}{\mathbin{|}}    

\newcommand{\salg}{\sigma \text{-}\mathrm{alg}}                  
\newcommand{\rh}{h^{\mathrm{Rok}}}                               
\newcommand{\ksh}{h^{\mathrm{KS}}}                               

\begin{document}

\title[Krieger's finite generator theorem for countable groups III]{Krieger's finite generator theorem for actions of countable groups III}
\author{Andrei Alpeev}
\address{Chebyshev lab at Saint Petersburg State University, 14th Line 29B, Vasilyevsky Island, St.Petersburg 199178, Russia}
\email{a.alpeev@spbu.ru}
\author{Brandon Seward}
\address{Courant Institute of Mathematical Sciences, New York University, 251 Mercer St, New York, NY 10012, USA}
\email{b.m.seward@gmail.com}

\begin{abstract}
We continue the study of Rokhlin entropy, an isomorphism invariant for {\pmp} actions of countable groups introduced in Part I. In this paper we prove a non-ergodic finite generator theorem and use it to establish sub-additivity and semi-continuity properties of Rokhlin entropy. We also obtain formulas for Rokhlin entropy in terms of ergodic decompositions and inverse limits. Finally, we clarify the relationship between Rokhlin entropy, sofic entropy, and classical Kolmogorov--Sinai entropy. In particular, using Rokhlin entropy we give a new proof of the fact that ergodic actions with positive sofic entropy have finite stabilizers.
\end{abstract}
\maketitle

\section{Introduction} \label{sec:intro}

Let $(X, \mu)$ be a standard probability space, meaning $X$ is a standard Borel space with Borel $\sigma$-algebra $\Borel(X)$ and $\mu$ is a Borel probability measure. Let $G$ be a countable group and let $G \acts (X, \mu)$ be a probability-measure-preserving (\pmp) action. For $\xi \subseteq \Borel(X)$ let $\salg(\xi)$ be the $\sigma$-algebra generated by $\xi$ and let $\salg_G(\xi)$ denote the smallest $G$-invariant $\sigma$-algebra containing $\xi$. A Borel partition $\alpha$ is \emph{generating}, or a \emph{generator}, if $\salg_G(\alpha) = \Borel(X)$ (equality is understood to be modulo $\mu$-null sets).

In Part I of this series, Seward defined the \emph{Rokhlin entropy} of a {\pmp} action $G \acts (X, \mu)$, denoted $\rh_G(X, \mu)$, to be
$$\inf \Big\{ \sH(\alpha \given \sinv_G) : \alpha \text{ a countable partition with } \salg_G(\alpha) \vee \sinv_G = \Borel(X) \Big\},$$
where $\sinv_G$ is the $\sigma$-algebra of $G$-invariant Borel sets and $\sH( \cdot \given \cdot)$ is conditional Shannon entropy (for the definition of conditional Shannon entropy see \cite[Def. 1.4.2]{Do11}). More generally, for a $G$-invariant sub-$\sigma$-algebra $\cF$, the \emph{Rokhlin entropy of $G \acts (X, \mu)$ relative to $\cF$}, denoted $\rh_G(X, \mu \given \cF)$, is
$$\inf \Big\{ \sH(\alpha \given \cF \vee \sinv_G) : \alpha \text{ a countable partition with } \salg_G(\alpha) \vee \cF \vee \sinv_G = \Borel(X) \Big\}.$$
In the special case of an ergodic action and trivial $\cF = \{\varnothing, X\}$, Rokhlin entropy simplifies to the more natural form
$$\rh_G(X, \mu) = \inf \Big\{ \sH(\alpha) : \alpha \text{ is a countable generating partition} \Big\}.$$

The purpose of this three-part series has been to introduce, motivate, and lay some basic foundations for Rokhlin entropy theory. Part I focused on ergodic actions and developed a generalization of Krieger's finite generator theorem for actions of arbitrary countable groups. We recall this theorem below as we will need to use it here.

\begin{thm}[\cite{S14}] \label{intro:krieger}
Let $G$ be a countably infinite group acting ergodically, but not necessarily freely, by measure-preserving bijections on a non-atomic standard probability space $(X, \mu)$. Let $\cF$ be a $G$-invariant sub-$\sigma$-algebra of $X$. If $\pv = (p_i)$ is any finite or countable probability vector with $\rh_G(X, \mu \given \cF) < \sH(\pv)$, then there is a partition $\alpha = \{A_i \: 0 \leq i < |\pv|\}$ with $\mu(A_i) = p_i$ for every $0 \leq i < |\pv|$ and $\salg_G(\alpha) \vee \cF = \Borel(X)$.
\end{thm}

The abstract nature of the definition of Rokhlin entropy, specifically an infimum over an extremely large set of partitions, initially seems to prevent any viable means of study. A hidden significance of the above theorem is that it changes this situation. Specifically, it leads to a sub-additive identity which unlocks a path to studying Rokhlin entropy.

Part II of this series motivated Rokhlin entropy theory through applications to Bernoulli shifts. The main theorem of Part II showed that a simple conjectured property of Rokhlin entropy would imply that the Bernoulli $2$-shift and the Bernoulli $3$-shift are non-isomorphic for all countably infinite groups, would positively solve the Gottschalk surjunctivity conjecture for all countable groups, and would positively solve the Kaplansky direct finiteness conjecture for all groups (all three of these are currently open problems). It's also worth noting that Rokhlin entropy theory has successfully been used to generalize the well known Sinai factor theorem to all countably infinite groups \cite{S16a}. Specifically, any free ergodic action of positive Rokhlin entropy factors onto all Bernoulli shifts of lesser or equal entropy. Rokhlin entropy has also been studied in \cite{Al,B15,GS,S16}.

Here in Part III, the final part of the series, we consider non-ergodic actions for the first time. Having introduced and motivated Rokhlin entropy in the prior parts, our goal here is to lay some basic foundations for the theory. A critical tool to doing this, the main theorem of this paper, is a generalization of Theorem \ref{intro:krieger} to non-ergodic actions. Recall that an action $G \acts (X, \mu)$ is \emph{aperiodic} if $\mu$-almost-every $G$-orbit is infinite. Below, for a Borel action $G \acts X$, we write $\E_G(X)$ for the set of $G$-invariant ergodic Borel probability measures on $X$.

\begin{thm} \label{intro:nekrieger}
Let $G \acts (X, \mu)$ be an aperiodic {\pmp} action and let $\cF$ be a countably generated $G$-invariant sub-$\sigma$-algebra. Let $\mu = \int_{\E_G(X)} \nu \ d \tau(\nu)$ be the ergodic decomposition of $\mu$. If $\nu \mapsto \pv^\nu$ is a Borel map associating to each $\nu \in \E_G(X)$ a finite or countable probability vector $\pv^\nu = (p_i^\nu)$ satisfying $\rh_G(X, \nu \given \cF) < \sH(\pv^\nu)$, then there is a partition $\alpha = \{A_i\}$ of $X$ such that $\salg_G(\alpha) \vee \cF = \Borel(X)$ and such that $\nu(A_i) = p_i^\nu$ for every $i$ and $\tau$-almost-every $\nu \in \E_G(X)$.
\end{thm}

This theorem is optimal in the sense that if $\salg_G(\alpha) \vee \cF = \Borel(X)$ (or if $\salg_G(\alpha) \vee \cF \vee \sinv_G = \Borel(X)$) then $\rh_G(X, \nu \given \cF) \leq \sH_\nu(\alpha)$ for $\tau$-almost-every $\nu \in \E_G(X)$ and in general there does not exist an $\alpha$ for which equality $\rh_G(X, \nu \given \cF) = \sH_\nu(\alpha)$ holds.

A nearly immediate consequence of this theorem is that Rokhlin entropy satisfies an ergodic decomposition formula.

\begin{cor} \label{intro:ergavg}
Let $G \acts (X, \mu)$ be a {\pmp} action, let $\cF$ be a countably generated $G$-invariant sub-$\sigma$-algebra, and let $\mu = \int_{\E_G(X)} \nu \ d \tau(\nu)$ be the ergodic decomposition of $\mu$. Then
$$\rh_G(X, \mu \given \cF) = \int_{\E_G(X)} \rh_G(X, \nu \given \cF) \ d \tau(\nu).$$
\end{cor}

Additionally, it follows that the formula for Rokhlin entropy simplifies in the setting of aperiodic actions.

\begin{cor} \label{intro:simprok}
Let $G \acts (X, \mu)$ be an aperiodic {\pmp} action and let $\cF$ be a $G$-invariant sub-$\sigma$-algebra. Then
$$\rh_G(X, \mu \given \cF) = \inf \Big\{ \sH(\alpha) : \alpha \text{ is a countable partition with } \salg_G(\alpha) \vee \cF = \Borel(X) \Big\}.$$
\end{cor}

As we mentioned before, currently the study of Rokhlin entropy is made possible by a sub-additive identity. This identity has played a crucial role in nearly all of the results on Rokhlin entropy appearing in the literature so far. This sub-additive identity was first proved for ergodic actions in Part I \cite{S14}, and generalized to countable sub-additivity in Part II \cite{S14b}. To state this property properly we need a definition. For a {\pmp} action $G \acts (X, \mu)$, a collection $\xi \subseteq \Borel(X)$, and a $G$-invariant sub-$\sigma$-algebra $\cF$, the \emph{outer Rokhlin entropy of $\xi$ relative to $\cF$}, denoted $\rh_{G,\mu}(\xi \given \cF)$, is defined to be
$$\inf \Big\{ \sH(\alpha \given \cF \vee \sinv_G) : \alpha \text{ a countable partition of $X$ with } \xi \subseteq \salg_G(\alpha) \vee \cF \vee \sinv_G \Big\}.$$
If $G \acts (Y, \nu)$ is a factor of $G \acts (X, \mu)$ and $\Sigma$ is the $G$-invariant sub-$\sigma$-algebra of $X$ associated to $Y$, then we define the \emph{outer Rokhlin entropy of $(Y, \nu)$ within $(X, \mu)$} to be $\rh_{G,\mu}(Y, \nu) = \rh_{G, \mu}(\Sigma)$.

Using Theorem \ref{intro:nekrieger} we prove countable sub-additivity for non-ergodic actions.

\begin{cor}[Countable sub-additivity of Rokhlin entropy] \label{intro:add2}
Let $G \acts (X, \mu)$ be a {\pmp} action, let $\cF$ be a $G$-invariant sub-$\sigma$-algebra, and let $\xi \subseteq \Borel(X)$. If $(\Sigma_n)_{n \in \N}$ is an increasing sequence of $G$-invariant sub-$\sigma$-algebras with $\xi \subseteq \bigvee_{n \in \N} \Sigma_n \vee \cF$ then
$$\rh_{G,\mu}(\xi \given \cF) \leq \rh_{G,\mu}(\Sigma_1 \given \cF) + \sum_{n = 2}^\infty \rh_{G,\mu}(\Sigma_n \given \Sigma_{n-1} \vee \cF).$$
In particular, if $G \acts (Y, \nu)$ is a factor of $(X, \mu)$ and $\Sigma$ is the sub-$\sigma$-algebra of $X$ associated to $Y$ then
$$\rh_G(X, \mu) \leq \rh_{G,\mu}(Y, \nu) + \rh_G(X, \mu \given \Sigma) \leq \rh_G(Y, \nu) + \rh_G(X, \mu \given \Sigma).$$
\end{cor}

Using sub-additivity, we show that Rokhlin entropy is a continuous function or an upper-semicontinuous function on a few natural spaces. Our work extends upon a few limited cases of upper-semicontinuity which were critical to the main theorems in \cite{S14b} and \cite{S16}. Recall that a function $f : X \rightarrow \R$ on a topological space $X$ is upper-semicontinuous if for every $r \in \R$ the set $f^{-1}((-\infty, r))$ is open. Below, for a Borel action $G \acts X$ we write $\M_G(X)$ for the set of $G$-invariant Borel probability measures, and we write $\M_G^{\text{aper}}(X)$ for the set of those measures $\mu \in \M_G(X)$ for which $G \acts (X, \mu)$ is aperiodic (and as before, $\E_G(X)$ is the set of ergodic measures).

\begin{cor}
Let $G$ be a countable group, let $L$ be a finite set, and let $L^G$ have the product topology. Let $\cF$ be a $G$-invariant sub-$\sigma$-algebra which is generated by a countable collection of clopen sets. Then the map $\mu \in \M_G^{\text{aper}}(L^G) \cup \E_G(L^G) \mapsto \rh_G(L^G, \mu \given \cF)$ is upper-semicontinuous in the weak$^*$-topology. Furthermore, if $G$ is finitely generated then this map is upper-semicontinuous on all of $\M_G(L^G)$.
\end{cor}

We also establish upper-semicontinuity results on the space of actions (Corollary \ref{cor:actups}), and we establish upper-semicontinuity and continuity results on the space of partitions (Corollary \ref{cor:pups} and Lemma \ref{lem:cont}).

Again relying upon sub-additivity, we develop a formula for the Rokhlin entropy of an inverse limit of actions. In the case of ergodic actions, this formula appeared in Part II and was a key ingredient to the proof of the main theorem there. See also Corollary \ref{cor:ks} for an alternate version of this theorem.

\begin{thm} \label{intro:ks}
Let $G \acts (X, \mu)$ be a {\pmp} action and let $\cF$ be a $G$-invariant sub-$\sigma$-algebra. Suppose that $G \acts (X, \mu)$ is the inverse limit of actions $G \acts (X_n, \mu_n)$. Identify each $\Borel(X_n)$ as a sub-$\sigma$-algebra of $X$ in the natural way. Then
\begin{equation*}
\rh_G(X, \mu \given \cF) < \infty \Longleftrightarrow
\left\{\begin{array}{c}
\displaystyle{\inf_{n \in \N} \sup_{m \geq n} \rh_{G,\mu}(\Borel(X_m) \given \Borel(X_n) \vee \cF) = 0}\\
\displaystyle{\text{and} \quad \forall m \ \rh_{G,\mu}(\Borel(X_m) \given \cF) < \infty.}
\end{array}\right\}\end{equation*}
Furthermore, when $\rh_G(X, \mu \given \cF) < \infty$ we have
\begin{equation*}
\rh_G(X, \mu \given \cF) = \sup_{m \in \N} \rh_{G,\mu}(\Borel(X_m) \given \cF) = \liminf_{m \rightarrow \infty} \rh_{G,\mu}(\Borel(X_m) \given \cF).
\end{equation*}
\end{thm}

It is unknown if $\rh_G(X, \mu \given \cF) = \sup_m \rh_{G,\mu}(\Borel(X_m) \given \cF)$ in all cases.

By using the inverse limit formula, we show that Rokhlin entropy is a Borel function on the space of $G$-invariant probability measures (Corollary \ref{cor:mborel}) and a Borel function on the space of {\pmp} $G$-actions (Corollary \ref{cor:aborel}).

We briefly observe that relative Rokhlin entropy is an invariant for certain restricted orbit equivalences. This generalizes a similar property of Kolmogorov--Sinai entropy discovered by Rudolph and Weiss \cite{RW00}. Unlike the original result by Rudolph and Weiss, for Rokhlin entropy this property follows quite easily from the definitions. Nevertheless, it feels worth explicitly mentioning.

\begin{prop} \label{intro:oe}
Let $G \acts (X, \mu)$ and $\Gamma \acts (X, \mu)$ be {\pmp} actions having the same orbits $\mu$-almost-everywhere. Let $\cF$ be a $G$-invariant and $\Gamma$-invariant sub-$\sigma$-algebra. Assume that there exist $\cF$-measurable maps $c_\Gamma : X \times G \rightarrow \Gamma$ and $c_G : X \times \Gamma \rightarrow G$ such that $g \cdot x = c_\Gamma(x, g) \cdot x$ and $\gamma \cdot x = c_G(x, \gamma) \cdot x$ for all $g \in G$, $\gamma \in \Gamma$, and $\mu$-almost-every $x \in X$. Then $\rh_G(X, \mu \given \cF) = \rh_\Gamma(X, \mu \given \cF)$.
\end{prop}

The final topic we consider is the relations between Rokhlin entropy, Kolmogorov--Sinai entropy, sofic entropy, and stabilizers.

In the case of standard (non-relative) entropies, it was Rokhlin who first showed that for free actions of $\Z$ Kolmogorov--Sinai entropy and Rokhlin entropy coincide \cite{Roh67} (the name `Rokhlin entropy' was chosen for this reason). Later this was extended to free ergodic actions of amenable groups by Seward and Tucker-Drob \cite{ST14}. Then in Part I \cite{S14} it was shown that relative Kolmogorov--Sinai entropy and relative Rokhlin entropy coincide for free ergodic actions of amenable groups. Here we completely settle the relationship by handling the non-ergodic case. This is an immediate consequence of the ergodic decomposition formula, Corollary \ref{intro:ergavg}. Below we write $\ksh$ for Kolmogorov--Sinai entropy. See Corollary \ref{cor:ksrok2} for a more refined version of this result involving outer Rokhlin entropy.

\begin{cor}
If $G \acts (X, \mu)$ is a free {\pmp} action of a countably infinite amenable group and $\cF$ is a $G$-invariant sub-$\sigma$-algebra, then $\rh_G(X, \mu \given \cF) = \ksh_G(X, \mu \given \cF)$.
\end{cor}

For sofic entropy, its precise relationship with Rokhlin entropy is still unclear. Specifically, it remains an important open problem to determine if Rokhlin entropy and sofic entropy coincide for free actions when the sofic entropy is not minus infinity. In any case, it is a fairly quick consequence of the definitions that sofic entropy is bounded above by Rokhlin entropy for ergodic actions (this follows from \cite[Prop. 5.3]{B10b} by letting $\beta$ be trivial). Combined with our Corollary \ref{intro:simprok}, \cite[Prop. 5.3]{B10b} in fact shows that sofic entropy is bounded above by Rokhlin entropy for aperiodic (not necessarily ergodic) actions. More generally, in \cite[Prop. 2.12]{H2} Hayes obtained a similar inequality, showing that for aperiodic actions the quantity he calls ``relative sofic entropy in the presence'' is bounded above by outer Rokhlin entropy. Formally, Hayes did not assume aperiodicity but he took the formula in our Corollary \ref{intro:simprok} as the definition of Rokhlin entropy, thus leaving open a technical gap in the case of actions that are not aperiodic. For the sake of having sound and complete literature, we close this gap. Below, for a sofic group $G$, a sofic approximation $\Sigma$ to $G$, a {\pmp} action $G \acts (X, \mu)$, and $G$-invariant sub-$\sigma$-algebras $\cF_1, \cF_2$, we write $h_{\Sigma,\mu}(\cF_1 \given \cF_2 \mathbin{:} X, G)$ for the $\Sigma$-sofic entropy of $\cF_1$ relative to $\cF_2$ in the presence of $X$, as defined by Hayes in \cite{H2}.

\begin{prop}
Let $G$ be a sofic group with sofic approximation $\Sigma$, let $G \acts (X, \mu)$ be a {\pmp} action, and let $\cF_1, \cF_2$ be $G$-invariant sub-$\sigma$-algebras. Then
$$h_{\Sigma,\mu}(\cF_1 \given \cF_2 \mathbin{:} X, G) \leq \rh_{G,\mu}(\cF_1 \given \cF_2).$$
In particular, the sofic entropy of $G \acts (X, \mu)$ is at most $\rh_G(X, \mu)$.
\end{prop}

Finally, we consider the effect of non-trivial stabilizers on entropy. It is a theorem of Meyerovitch that ergodic actions of positive sofic entropy must have finite stabilizers \cite{Me15}. For Rokhlin entropy this is certainly not the case. If $G \acts (X, \mu)$ is a {\pmp} action and $G$ is a quotient of $\Gamma$, then $\Gamma$ acts on $(X, \mu)$ by factoring through $G$, and it is easily checked that $\rh_\Gamma(X, \mu) = \rh_G(X, \mu)$. Nevertheless, outer Rokhlin entropy can detect when new stabilizers appear in a factor action.

\begin{thm}
Let $G \acts (X, \mu)$ be an aperiodic {\pmp} action. Consider a factor $f : G \acts (X, \mu) \rightarrow G \acts (Y, \nu)$.
\begin{enumerate}
\item[\rm (i)] If $|\Stab_G(f(x)) : \Stab_G(x)| \geq k$ for $\mu$-almost-every $x \in X$ then
$$\rh_{G,\mu}(Y, \nu) \leq \frac{1}{k} \cdot \rh_G(Y, \nu).$$
\item[\rm (ii)] If $|\Stab_G(f(x)) : \Stab_G(x)| = \infty$ for $\mu$-almost-every $x \in X$ then
$$\rh_{G,\mu}(Y, \nu) = 0.$$
\end{enumerate}
\end{thm}

As a consequence of this theorem, we obtain a new relationship between Rokhlin entropy and sofic entropy for non-free actions. This also provides a new proof of Meyerovitch's theorem \cite{Me15} which stated that ergodic actions of positive sofic entropy must have finite stabilizers.

\begin{cor}
Let $G$ be a sofic group with sofic approximation $\Sigma$ and let $G \acts (X, \mu)$ be a {\pmp} action.
\begin{enumerate}
\item[\rm (i)] If $\mu$-almost-every stabilizer has cardinality at least $k \in \N$, then
$$h_G^\Sigma(X, \mu) \leq \frac{1}{k} \cdot \rh_G(X, \mu).$$
\item[\rm (ii)] If $\mu$-almost-every stabilizer is infinite then
$$h_G^\Sigma(X, \mu) = 0.$$
\end{enumerate}
\end{cor}

\subsection*{Acknowledgments}
Andrei Alpeev was supported by ``Native towns,'' a social investment program of PJSC ``Gazprom Neft''  and by St. Petersburg State University grant 6.37.208.2016. Brandon Seward was partially supported by the NSF Graduate Student Research Fellowship under Grant No. DGE 0718128, NSF RTG grant 1045119, ERC grant 306494, and Simons Foundation grant 328027 (P.I. Tim Austin)

\section{Measurable selection}

Our first goal is to prove the main theorem, Theorem \ref{intro:nekrieger}. This task will occupy the next three sections. Let us briefly outline the proof. Fix an action $G \acts (X, \mu)$, a $G$-invariant sub-$\sigma$-algebra $\cF$, and a Borel map $\nu \mapsto \pv^\nu$ for $\nu \in \E_G(X)$ as described by the theorem. By the ergodic decomposition theorem (recorded below in Lemma \ref{lem:farvar}) there is a $G$-invariant Borel partition $\{X_\nu : \nu \in \E_G(X)\}$ of $X$ satisfying $\nu(X_\nu) = 1$ for all $\nu \in \E_G(X)$. By Theorem \ref{intro:krieger} from Part I, for every $\nu$ there is a partition $\alpha^\nu = \{A_i^\nu : i \in \N\}$ of $X_\nu$ satisfying $\dist_\nu(\alpha^\nu) = \pv^\nu$ and $\salg_G(\alpha_\nu) \vee \cF = \Borel(X_\nu)$ (modulo $\nu$-null sets). If we define $\alpha = \{A_i : i \in \N\}$ where $A_i = \bigcup_\nu A_i^\nu$, then at first it may seem that $\alpha$ has the desired properties. However, a problem is that $\alpha$ may not be Borel. In order to fix this and complete the proof, we need to ensure that the map $\nu \mapsto \alpha^\nu$ is (in some sense) Borel.

In this section we digress into pure descriptive set theory. We consider the problem of choosing Borel maps satisfying certain restrictions. Later this will be applied for choosing the map $\nu \mapsto \alpha^\nu$.

Recall that a subset $B$ of a standard Borel space $X$ is \emph{analytic} if it is the image of a Borel set under a Borel map. Below, for a set $A \subseteq X \times Y$ we denote the cross-section of $A$ above $x \in X$ by $A_x = \{y \in Y : (x, y) \in A\}$.

\begin{lem} \label{lem:uncount}
Let $(X, \mu)$ be a standard probability space, let $Y$ be a standard Borel space, and let $A \subseteq X \times Y$ be an analytic set such that $A_x$ is uncountable for every $x \in X$. Then there is a Borel set $\bar{A} \subseteq A$ such that $\bar{A}_x$ is uncountable for $\mu$-almost-every $x \in X$.
\end{lem}

\begin{proof}
This is trivial if $X$ is countable, so we may assume $X$ is uncountable. By our assumptions $Y$ is also uncountable, so without loss of generality we may assume $X = \N^\N$ and $Y = \N^\N$ are the Baire space. Also let $Z = \N^\N$ be another copy of the Baire space. For $m \geq 1$ let $\N^m$ be the set of sequences of natural numbers of length $m$, and for $m = 0$ write $\N^0$ for set consisting of the empty sequence $\varnothing$. When $m < n$ and $y \in \N^n$ or $y \in \N^\N$, we let $y \res m \in \N^m$ denote the length $m$ prefix of $y$ (i.e. the first $m$ terms of $y$). For each $\N^m$ and for $\N^\N$ we write $<$ for the lexicographic order. For $s \in \N^m$, let $Y_s$ be the set of $y \in Y = \N^\N$ having $s$ as a prefix. Define $Z_s \subseteq Z$ similarly.

Since $A \subseteq X \times Y$ is analytic, it is equal to the projection $\pi_{X \times Y}(B)$ of a closed set $B \subseteq X \times Y \times Z$ \cite[Prop. 25.2]{K95}. For $s, t \in \N^m$ let $X_{s,t}$ be the set of $x \in X$ such that $\pi_Y((\{x\} \times Y_s \times Z_t) \cap B)$ is uncountable. Then $X_{s,t}$ is analytic \cite[Theorem 29.19]{K95} and so we may fix Borel sets $X'_{s,t} \subseteq X_{s,t} \subseteq X''_{s,t}$ such that $\mu(X''_{s,t} \setminus X'_{s,t}) = 0$ \cite[Theorem 21.10]{K95}. Set
$$\bar{X} = X \setminus \bigcup_{m \in \N} \bigcup_{s,t \in \N^m} (X''_{s,t} \setminus X'_{s,t})$$
and set $\bar{X}_{s,t} = X'_{s,t} \cap \bar{X}$. Note that $\bar{X}_{s,t}$ and $\bar{X}$ are Borel and that $\bar{X}$ is conull. Also note that for $x \in \bar{X}$ we have that $\pi_Y((\{x\} \times Y_s \times Z_t) \cap B)$ is uncountable if and only if $x \in \bar{X}_{s,t}$.

In the remainder of the proof we will build a Borel set $\bar{B} \subseteq B$ with the property that $\pi_{X \times Y}$ is injective on $\bar{B}$ and $\bar{B}_x$ is uncountable for every $x \in \bar{X}$. We remark to the familiar reader that our argument is essentially an explicit description, fibered over $\bar{X}$, of a winning strategy for Player I in the usual unfolded cut-and-choose game (see \cite[Sec. 21.B]{K95}). 

For $m \in \N$ let $P_m$ be the set of triples $(x, s, t) \in X \times \N^m \times \N^m$ such that $x \in \bar{X}_{s,t}$ but $x \not\in \bar{X}_{s,t'}$ whenever $t' \in \N^m$ with $t' < t$. We claim that if $(x, s, t) \in P_m$ then there is $n > m$, $s_1 \neq s_2 \in \N^n$ extending $s$, and $t_1, t_2 \in \N^n$ extending $t$ such that $(x, s_i, t_i) \in P_n$. By definition $x \in \bar{X}_{s,t}$ implies that $\pi_Y((\{x\} \times Y_s \times Z_t) \cap B)$ is uncountable. So there is $n > m$ and $s_1 \neq s_2 \in \N^n$ extending $s$ such that each set $\pi_Y((\{x\} \times Y_{s_i} \times Z_t) \cap B)$ is uncountable. Now for $i = 1, 2$ let $t_i \in \N^n$ be the least extension of $t$ with $\pi_Y((\{x\} \times Y_{s_i} \times Z_{t_i}) \cap B)$ uncountable. If $t' \in \N^n$ and $t' < t_i$, then either $(t' \res m) < t$ or $t'$ extends $t$. In either case we will have that $\pi_Y((\{x\} \times Y_{s_i} \times Z_{t'}) \cap B)$ is at most countable and thus $x \not\in \bar{X}_{s_i, t'}$. So $(x, s_i, t_i) \in P_n$, completing the claim.

Now define
$$\bar{B} = \{(x, y, z) : \forall k \ \exists m \geq k \ (x, y \res m, z \res m) \in P_m\}.$$
If $(x, y, z) \in \bar{B}$ then $x \in \bar{X}_{y \res m, z \res m}$ for infinitely many $m$. Since $B$ is closed, it follows that $(x, y, z) \in B$. So $\bar{B}$ is a Borel subset of $B$. If $(x, y, z) \in \bar{B}$ and $z' < z$, then there is $m$ with $(x, y \res m, z \res m) \in P_m$ and $z' \res m < z \res m$ and therefore the definition of $P_m$ gives $x \not\in \bar{X}_{y \res m, z' \res m}$. This implies that $(x, y, z') \not\in \bar{B}$. So the restriction of $\pi_{X \times Y}$ to $\bar{B}$ is injective and hence $\bar{A} = \pi_{X \times Y}(\bar{B})$ is a Borel subset of $A$ \cite[Cor. 15.2]{K95}. To complete the proof, we claim that for every $x \in \bar{X}$ the set $\bar{A}_x$ is uncountable. As $\pi_{X \times Y} : \bar{B} \rightarrow \bar{A}$ is injective, it suffices to show that $\bar{B}_x$ is uncountable for $x \in \bar{X}$. Indeed, fixing $x \in \bar{X}$, we can construct a collection of pairs $(s_v,t_v)_{v \in \{ 0, 1\}^{<\N}}$ indexed by the rooted binary tree $\{ 0, 1\}^{<\N}$ in such a way that $(\varnothing, \varnothing )$ is assigned to the root, and the children of each vertex are assigned according to the claim from the previous paragraph. It is not hard to see that for any infinite path $p$ in the tree, the intersection $\bigcap_{v \in p} (Y_{s_v} \times Z_{s_v})$ is one-point and belongs to $\bar{B}_x$. So to each infinite path a point from $\bar{B}_x$ is assigned. We also note that this path-point assignment is one-to-one since if $u$ and $v$ are two children of the same vertex, then by the construction and the claim, we have $s_u \neq s_v$, so the sets $Y_{s_u} \times Z_{t_u}$ and $Y_{s_v} \times Z_{t_v}$ are disjoint. This means that $\bar{B}_x$ contains a copy of the Cantor space, thus finishing the proof.
\end{proof}

The previous lemma gives an improved version of an injective selection theorem due to Graf and Mauldin \cite{GrMa}.

\begin{cor} \label{cor:inject}
Let $(X, \mu)$ be a standard probability space, let $Y$ be a standard Borel space, and let $A \subseteq X \times Y$ be an analytic set such that $A_x$ is uncountable for every $x \in X$. Then there is a conull Borel set $X' \subseteq X$ and a Borel injection $f : X' \rightarrow Y$ whose graph is contained in $A$.
\end{cor}

\begin{proof}
If $A$ is in fact Borel then this is a special case of a theorem of Graf and Mauldin \cite{GrMa}. By applying Lemma \ref{lem:uncount} we obtain a conull Borel set $X'' \subseteq X$ and a Borel set $\bar{A} \subseteq A$ with $\bar{A}_x$ uncountable for every $x \in X''$. Now apply the Graf--Mauldin theorem to obtain a conull Borel set $X' \subseteq X''$ and a Borel injection $f : X' \rightarrow Y$ whose graph is contained in $\bar{A} \subseteq A$.
\end{proof}

Finally, we state the descriptive set theory result which we will need for proving Theorem \ref{intro:nekrieger}.

\begin{prop} \label{prop:select}
Let $(X, \mu)$ be a standard probability space, let $Y$ and $Z$ be standard Borel spaces, let $f : X \times Y \rightarrow Z$ be Borel, and let $A \subseteq X \times Y$ be Borel with $f(\{x\} \times A_x)$ uncountable for every $x \in X$. Then there is a conull Borel set $X'$ and a Borel function $\phi : X' \rightarrow Y$ whose graph is contained in $A$ such that the map $x \in X' \mapsto f(x, \phi(x))$ is injective.
\end{prop}

\begin{proof}
Set $B = \{(x, f(x,y)) : (x, y) \in A\} \subseteq X \times Z$. Then $B$ is analytic and $B_x$ is uncountable for every $x \in X$. Apply Corollary \ref{cor:inject} to obtain a conull Borel set $X'' \subseteq X$ and a Borel injection $\psi : X'' \rightarrow Z$ whose graph is contained in $B$.

Let $A' \subseteq A$ be the set of $(x, y) \in A$ with $x \in X''$ and $f(x, y) = \psi(x)$. Then $A'$ is Borel and $A'_x \neq \varnothing$ for all $x \in X''$. By the Jankov--von Neumann uniformization theorem \cite[Theorems 29.9]{K95}, there is a $\mu$-measurable (but possibly not Borel measurable) function $\phi_0 : X'' \rightarrow Y$ whose graph is contained in $A'$. Since $\phi_0$ is $\mu$-measurable and $Y$ is standard Borel, there exists a conull Borel set $X' \subseteq X''$ such that the restriction $\phi = \phi_0 \res X'$ is Borel measurable (take a countable collection of sets generating the Borel $\sigma$-algebra of $Y$, and for each such set we can make its preimage be Borel by removing a Borel null set from $X''$). The graph of $\phi$ is still contained in $A' \subseteq A$ and the map $x \in X' \mapsto f(x, \phi(x)) = \psi(x)$ is injective.
\end{proof}

\section{Ergodic components and Bochner measurability}

For a Borel action $G \acts X$ on a standard Borel space $X$, we write $\M_G(X)$ for the set of $G$-invariant Borel probability measures and $\E_G(X) \subseteq \M_G(X)$ for the ergodic measures. Recall that both $\M_G(X)$ and $\E_G(X)$ are standard Borel spaces. Their Borel $\sigma$-algebras are defined by requiring the map $\mu \mapsto \mu(A)$ to be Borel measurable for every Borel set $A \subseteq X$.

For a standard probability space $(X, \mu)$ we write $\Null_\mu$ for the $\sigma$-ideal of $\mu$-null Borel sets. For Borel sets $A, B \subseteq X$ we write $A = B \mod \Null_\mu$ if $A \symd B \in \Null_\mu$. Similarly for $\sigma$-algebras $\cF, \Sigma \subseteq \Borel(X)$ we write $\cF \subseteq \Sigma \mod \Null_\mu$ if for every $A \in \cF$ there is $B \in \Sigma$ with $A = B \mod \Null_\mu$. When $\cF \subseteq \Sigma \mod \Null_\mu$ and $\Sigma \subseteq \cF \mod \Null_\mu$ we write $\cF = \Sigma \mod \Null_\mu$. We will only write ``$\!\!\! \mod \Null_\mu$'' to add clarity and emphasis, but frequently we will omit this notation when it is clear from context.

We say that a sub-$\sigma$-algebra $\cF$ is \emph{countably generated} if there is a countable collection $\xi \subseteq \Borel(X)$ with $\cF = \salg(\xi)$ (a literal equality without discarding any null sets). It is well known that $\Borel(X)$ is countably generated when $X$ is standard Borel. Moreover, if $\mu$ is a Borel probability measure on $X$ and $\cF \subseteq \Borel(X)$ is a $\sigma$-algebra, then there is a countably generated $\sigma$-algebra $\cF'$ with $\cF = \cF' \mod \Null_\mu$. Thus being countably generated is vacuously true modulo null sets. However, working with countably generated $\sigma$-algebras is vital when we consider ergodic decompositions, for otherwise strange things can happen. For example if $\mu \in \M_G(X)$ has continuum-many ergodic components and $\cF = \salg(\Null_\mu)$, then for almost-every ergodic component $\nu$ of $\mu$ we will have $\cF = \Borel(X) \mod \Null_\nu$. We will also need to use the following lemma.

\begin{lem} \label{lem:break}
Let $G \acts (X, \mu)$ be a {\pmp} action, let $\cF$ be a countably generated sub-$\sigma$-algebra, and let $\mu = \int_{\E_G(X)} \nu \ d \tau(\nu)$ be the ergodic decomposition of $\mu$. Then for every countable Borel partition $\xi$ of $X$ we have
$$\sH_\mu(\xi \given \cF \vee \sinv_G) = \int_{\E_G(X)} \sH_\nu(\xi \given \cF) \ d \tau(\nu).$$
\end{lem}

\begin{proof}
This is likely well known. See \cite[Lem. 2.2]{S16} for a short proof of a slightly more general fact.
\end{proof}

We also need the following uniform ergodic decomposition theorem.

\begin{lem}[Farrell \cite{Far62}, Varadarajan \cite{Var63}] \label{lem:farvar}
Let $X$ be a standard Borel space, let $G$ be a countable group, and let $G \acts X$ be a Borel action. Assume that $\M_G(X) \neq \varnothing$. Then there is a Borel surjection $x \mapsto \nu_x$ from $X$ onto $\E_G(X)$ such that
\begin{enumerate}
\item[\rm (1)] if $x$ and $y$ are in the same orbit then $\nu_x = \nu_y$,
\item[\rm (2)] for each $\nu \in \E_G(X)$ we have $\nu(\{x \in X : \nu_x = \nu\}) = 1$, and
\item[\rm (3)] for each $\mu \in \M_G(X)$ we have $\mu = \int_{x \in X} \nu_x \ d \mu(x) = \int_{\nu \in \E_G(X)} \nu \ d \tau(\nu)$, where $\tau$ is the push-forward of $\mu$ under the map $x \mapsto \nu_x$.
\end{enumerate}
\end{lem}

In the proof outline discussed at the beginning of the previous section, we mentioned that we wanted a ``Borel'' map associating to each ergodic measure $\nu$ a partition $\alpha^\nu$. A priori it is not clear how to represent this as a map from $\E_G(X)$ to some standard Borel space. In this section we lay down the technical framework which will allow us to do so. We will also use an auxiliary notion of Bochner measurability and record some useful applications.

\begin{defn}
Let $X$ and $Y$ be standard Borel spaces, let $G \acts X$ be a Borel action, and let $E \subseteq \E_G(X)$ be Borel. We say that a function $f : E \times Y \rightarrow \Borel(X)$ is \emph{Bochner measurable} if there exists a sequence of countably-valued Borel functions (i.e. the pre-image of every point is a Borel set) $f_n : E \times Y \rightarrow \Borel(X)$ such that $\lim_{n \rightarrow \infty} \nu(f(\nu, y) \symd f_n(\nu, y)) = 0$ for all $(\nu, y) \in E \times Y$. If $\tau$ is a probability measure on $\E_G(X)$, we say $f : \E_G(X) \times Y \rightarrow \Borel(X)$ is Bochner measurable \emph{$\tau$-almost-everywhere} if there is a $\tau$-conull set $E \subseteq \E_G(X)$ such that the restriction of $f$ to $E \times Y$ is Bochner measurable.
\end{defn}

\begin{lem} \label{lem:adalg}
Let $X$ and $Y$ be standard Borel spaces and let $G \acts X$ be a Borel action. Then the set of Bochner measurable functions $f : \E_G(X) \times Y \rightarrow \Borel(X)$ form a coordinate-wise $G$-invariant algebra. More specifically, if $f, k : \E_G(X) \times Y \rightarrow \Borel(X)$ are Bochner measurable, then so are the functions sending $(\nu, y)$ to $X \setminus f(\nu, y)$, $f(\nu, y) \cup k(\nu, y)$, $f(\nu, y) \cap  k(\nu, y)$, or $g \cdot f(\nu, y)$ (for any fixed $g \in G$). Furthermore, every constant function from $\E_G(X) \times Y$ to $\Borel(X)$ is Bochner measurable.
\end{lem}

\begin{proof}
Let $f_n, k_n : \E_G(X) \times Y \rightarrow \Borel(X)$ be the sequence of functions as described in the definition of Bochner measurability.
Apply the same operations to $f_n, k_n$, and recall that each $\nu$ is $G$-invariant. The final claim is immediate from the definition.
\end{proof}

\begin{lem} \label{lem:mm}
Let $X$ and $Y$ be standard Borel spaces and let $G \acts X$ be a Borel action. If $f : \E_G(X) \times Y \rightarrow \Borel(X)$ is Bochner measurable, then the map $(\nu, y) \mapsto \nu(f(\nu, y))$ is Borel.
\end{lem}

\begin{proof}
Let $f_n : \E_G(X) \times Y \rightarrow \Borel(X)$ be as in the definition of Bochner measurability. For each $n$ the function $f_n$ is countably-valued and Borel. Since for fixed $A \in \Borel(X)$ the map $\nu \mapsto \nu(A)$ is Borel, it follows that $(\nu, y) \mapsto \nu(f_n(\nu, y))$ is Borel. Therefore $\nu(f(\nu, y)) = \lim_{n \rightarrow \infty} \nu(f_n(\nu, y))$ is Borel.
\end{proof}

Our interest in Bochner measurable functions comes from the following lemma.

\begin{lem} \label{lem:fibers}
Let $X$ be a standard Borel space, let $G \acts X$ be a Borel action, let $f : \E_G(X) \rightarrow \Borel(X)$ be Bochner measurable, and let $\Sigma \subseteq \Borel(X)$ be a countably generated sub-$\sigma$-algebra. If $f(\nu) \in \Sigma \mod \Null_\nu$ for every $\nu \in \E_G(X)$, then there is a Borel set $B \in \sinv_G \vee \Sigma$ which satisfies $\nu(B \symd f(\nu)) = 0$ for every $\nu \in \E_G(X)$. In particular, if $A \in \Sigma \mod \Null_\nu$ for every $\nu \in \E_G(X)$ then $A \in \sinv_G \vee \Sigma \mod \Null_\mu$ for every $\mu \in \M_G(X)$.
\end{lem}

\begin{proof}
Fix a countable algebra $\{C_0, C_1, \ldots\}$ which generates $\Sigma$. Fix $\epsilon > 0$ and for $n \in \N$ let $D_{n, \epsilon}$ be the set of $\nu$ such that $n$ is least with $\nu(C_n \symd f(\nu)) < \epsilon$. Then $\{D_{n,\epsilon} : n \in \N\}$ is a Borel partition of $\E_G(X)$ by Lemmas \ref{lem:adalg} and \ref{lem:mm} since $\nu \mapsto C_n \symd f(\nu)$ is Bochner measurable. Let $\{D_{n, \epsilon}' : n \in \N\}$ be the $\sinv_G$-measurable partition of $X$ associated to $\{D_{n,\epsilon} : n \in \N\}$ by the ergodic decomposition. Define $B_\epsilon = \bigcup_{n \in \N} (D_{n, \epsilon}' \cap C_n)$. The Borel--Cantelli lemma implies that $B = \bigcup_{k \in \N} \bigcap_{m \geq k} B_{2^{-m}}$ satisfies $\nu(B \symd f(\nu)) = 0$ for all $\nu \in \E_G(X)$. Also, $B \in \sinv_G \vee \Sigma$ as claimed. The final claim also follows by using the Bochner measurable (constant) function $f(\nu) = A$.
\end{proof}

The next lemma introduces a useful $\sigma$-algebra $\mathcal{M} \subseteq \Borel(X)$.

\begin{lem} \label{lem:split}
Let $G \acts (X, \mu)$ be an aperiodic {\pmp} action. Let $\mu = \int_{\E_G(X)} \nu \ d \tau(\nu)$ be the ergodic decomposition of $\mu$.
Then there is a countably generated sub-$\sigma$-algebra $\mathcal{M} \subseteq \Borel(X)$ such that
\begin{enumerate}
\item[\rm (i)] for $\tau$-almost-every $\nu \in \E_G(X)$ we have $\mathcal{M} = \Borel(X) \mod \Null_\nu$;
\item[\rm (ii)] for $\tau$-almost-every $\nu \in \E_G(X)$ and for every $B \in \mathcal{M}$ we have $\nu(B) = \mu(B)$;
\item[\rm (iii)] $\sinv_G \vee \mathcal{M} = \Borel(X) \mod \Null_\mu$.
\end{enumerate}
\end{lem}

\begin{proof}
Let $\phi : (X, \mu) \rightarrow (\E_G(X), \tau)$ be the ergodic decomposition map given by Lemma \ref{lem:farvar}. Aperiodicity implies that the fiber measures $\nu \in \E_G(X)$ are non-atomic. So the Rokhlin skew-product theorem \cite[Theorem 3.18]{Gl03}\footnote{The proof of \cite[Theorem 3.18]{Gl03} assumes ergodicity, however that assumption is only used to conclude that all fiber measures have the same number of atoms. Since all fiber measures are non-atomic in our case, we can apply this result without requiring ergodicity. See also \cite{Rok57}.} implies that there is a measure space isomorphism $\psi : (X, \mu) \rightarrow (\E_G(X) \times [0, 1], \tau \times \lambda)$, where $\lambda$ is Lebesgue measure, such that $\phi$ equals $\psi$ composed with the projection map. View $\Borel([0, 1]) \subseteq \Borel(\E_G(X) \times [0, 1])$ in the natural way and set $\mathcal{M} = \psi^{-1}(\Borel([0, 1]))$. Clauses (i) and (ii) are satisfied since $\psi_*(\nu) = \delta_\nu \times \lambda$ for $\tau$-almost-every $\nu$, and (iii) is satisfied since $\psi^{-1}(\Borel(\E_G(X))) = \sinv_G \mod \Null_\mu$.
\end{proof}

The $\sigma$-algebra $\mathcal{M}$ allows us to return to discussing Borel measurable functions, while still being able to use Bochner measurability. Below we write $\mu \res \mathcal{M}$ for the restriction of $\mu$ to the $\sigma$-algebra $\mathcal{M}$, and we let $\malg_{\mu \res \mathcal{M}}$ denote the corresponding measure-algebra. Specifically, $\malg_{\mu \res \mathcal{M}}$ consists of the classes $[A]_\mu$, where $A \in \mathcal{M}$ and $[A]_\mu = \{B \in \mathcal{M} : \mu(A \symd B) = 0\}$, equipped with the Polish topology induced by the complete separable metric $\dB([A]_\mu, [B]_\mu) = \mu(A \symd B)$. In particular, $\malg_{\mu \res \mathcal{M}}$ is a standard Borel space. The operations of union, intersection, complement and the function $\mu$ clearly descend to $\malg_{\mu \res \mathcal{M}}$. Furthermore, by Lemma \ref{lem:split}.(ii) $\tau$-almost-every $\nu \in \E_G(X)$ descends to $\malg_{\mu \res \mathcal{M}}$ and in fact coincides with $\mu$ on this space.

\begin{lem} \label{lem:mad}
Let $G \acts (X, \mu)$, $\tau$, and $\mathcal{M}$ be as in Lemma \ref{lem:split}, and let $Y$ be a standard Borel space. If $f : \E_G(X) \times Y \rightarrow \malg_{\mu \res \mathcal{M}}$ is Borel, then any function $\bar{f} : \E_G(X) \times Y \rightarrow \mathcal{M}$ satisfying $[\bar{f}(\nu, y)]_\mu = f(\nu, y)$ is Bochner measurable $\tau$-almost-everywhere.
\end{lem}

\begin{proof}
Let $\{[M_0]_\mu, [M_1]_\mu, \ldots\}$ be a countable dense subset of $\malg_{\mu \res \mathcal{M}}$. Define $q_n : \malg_{\mu \res \mathcal{M}} \rightarrow \{M_0, M_1, \ldots\}$ by setting $q_n([B]_\mu) = M_k$ if $k$ is least with $\mu(B \symd M_k) < 1 / n$. Then $q_n \circ f : \E_G(X) \times Y \rightarrow \{M_0, M_1, \ldots\}$ is a countable-valued Borel function. Finally, for $\tau$-almost-every $\nu \in \E_G(X)$ and every $y \in Y$, Lemma \ref{lem:split}.(ii) gives
\begin{align*}
\lim_{n \rightarrow \infty} \nu(\bar{f}(\nu, y) \symd q_n \circ f(\nu, y)) & = \lim_{n \rightarrow \infty} \mu(\bar{f}(\nu, y) \symd q_n \circ f(\nu, y))\\
 & = \lim_{n \rightarrow \infty} \mu(f(\nu, y) \symd q_n \circ f(\nu, y)) = 0.\qedhere
\end{align*}
\end{proof}

Lastly, we consider sufficient conditions for a partition $\alpha$ to satisfy $\salg_G(\alpha) \supseteq \sinv_G$. For a partition $\alpha = \{A_i : i \in \N\}$, define $\theta^\alpha : X \rightarrow \N^G$ by the rule $\theta^\alpha(x)(g) = i \Leftrightarrow g^{-1} \cdot x \in A_i$. Note that $\theta^\alpha$ is $G$-equivariant, where the action of $G$ on $\N^G$ is given by $(g \cdot y)(t) = y(g^{-1} \cdot t)$ for $y \in \N^G$ and $g, t \in G$.

\begin{lem} \label{lem:sinv}
Let $G \acts (X, \mu)$ be a {\pmp} action and let $\mu = \int_{\E_G(X)} \nu \ d \tau(\nu)$ be the ergodic decomposition of $\mu$. If $\alpha = \{A_i : i \in \N\}$ is a partition and the map $\nu \in \E_G(X) \mapsto \theta^\alpha_*(\nu)$ is injective on a $\tau$-conull set then $\sinv_G \subseteq \salg_G(\alpha) \mod \Null_\mu$.
\end{lem}

\begin{proof}
Fix a symmetric probability measure $\lambda$ on $G$ whose support generates $G$, and let $\lambda^{*n}$ denote the $n^{\text{th}}$ convolution power of $\lambda$. Let $\pi : (X, \mu) \rightarrow (\E_G(X), \tau)$ be the ergodic decomposition map given by Lemma \ref{lem:farvar}. For any finite $T \subseteq G$ and function $f : T \rightarrow \N$ set $A_f = \bigcap_{t \in T} t \cdot A_{f(t)}$ and set $B_f = \{y \in \N^G : \forall t \in T \ y(t) = f(t)\}$. By the Kakutani ergodic theorem \cite{Ka51} (see also the work of Oseledets \cite{Os65}) we have that for $\mu$-almost-every $x \in X$
$$\theta^\alpha_*(\pi(x))(B_f) = \lim_{n \rightarrow \infty} \sum_{g \in G} \lambda^{*n}(g) \cdot \chi_{A_f}(g \cdot x),$$
where $\chi_C$ is the indicator function for $C \subseteq X$. It follows that the map $x \mapsto \theta^\alpha_*(\pi(x))$ is $\salg_G(\alpha) \vee \Null_\mu$-measurable. By assumption, there is a $\tau$-conull Borel set $E \subseteq \E_G(X)$ so that $\theta^\alpha_*$ is injective on $E$. It follows that there exists a Borel function $\psi : \E_G(\N^G) \rightarrow \E_G(X)$ such that $\psi \circ \theta^\alpha_*$ is the identity map on $E$. So $\pi(x) = \psi \circ \theta^\alpha_*(\pi(x))$ for $\mu$-almost-every $x \in X$, and thus $\pi$ is $\salg_G(\alpha) \vee \Null_\mu$-measurable. We conclude that $\sinv_G \subseteq \salg_G(\alpha) \mod \Null_\mu$.
\end{proof}

\begin{cor} \label{cor:invzero}
Let $G \acts (X, \mu)$ be an aperiodic {\pmp} action. Then for every $\epsilon > 0$ there is a two-piece partition $\alpha$ of $X$ with $\sH(\alpha) < \epsilon$ and with $\sinv_G \subseteq \salg_G(\alpha)$.
\end{cor}

We mention that in the purely Borel context, a result quite similar to the above corollary was obtained by Tserunyan \cite[Thm. 8.12]{T15}.

\begin{proof}
Fix $\epsilon > 0$ and let $0 < \delta < 1/2$ be small enough that $- \delta \cdot \log(\delta) - (1 - \delta) \cdot \log(1 - \delta) < \epsilon$. Let $\mathcal{M}$ be as in Lemma \ref{lem:split} and note that clauses (i) and (ii) of that lemma together with aperiodicity of the action imply that $\mu \res \mathcal{M}$ is non-atomic. Fix a Borel injection $\iota : \E_G(X) \rightarrow (0, \delta)$ and fix a Borel map $\psi : (0, \delta) \rightarrow \malg_{\mu \res \mathcal{M}}$ satisfying $\mu(\psi(t)) = t$ for all $0 < t < \delta$ (it is easily seen from the construction of $\mathcal{M}$ in Lemma \ref{lem:split} that such a Borel map exists). Set $f = \psi \circ \iota$. By Lemmas \ref{lem:fibers} and \ref{lem:mad} there is a Borel set $A$ with $\nu(A \symd f(\nu)) = 0$ for all $\nu \in \E_G(X)$. Set $\alpha = \{A, X \setminus A\}$. Note that $\nu(A) = \nu(f(\nu)) = \mu(f(\nu))$. Therefore $\mu(A) < \delta$ and $\sH(\alpha) < \epsilon$. Furthermore $\nu \mapsto \nu(A)$ is injective, so the map $\nu \mapsto \theta^\alpha_*(\nu)$ is injective as well. Thus $\sinv_G \subseteq \salg_G(\alpha) \mod \Null_\mu$ by Lemma \ref{lem:sinv}.
\end{proof}

\section{Generating partitions for non-ergodic actions} \label{sec:erg}

In this section we prove the main theorem, Theorem \ref{intro:nekrieger}. We will need to rely on the following strong form of the main theorem from Part I.

\begin{thm}[\cite{S14}] \label{thm:p1strong}
Let $G \acts (X, \mu)$ be an aperiodic ergodic {\pmp} action, let $\xi \subseteq \Borel(X)$, and let $\cF$ be a $G$-invariant sub-$\sigma$-algebra. If $0 < r \leq 1$ and $\pv = (p_i)$ is a probability vector satisfying $\rh_{G,\mu}(\xi \given \cF) < r \cdot \sH(\pv)$, then there is a collection $\alpha^* = \{A_i^* : 0 \leq i < |\pv|\}$ of pairwise disjoint Borel sets such that $\mu(A_i^*) = r \cdot p_i$ and such that $\xi \subseteq \salg_G(\alpha) \vee \cF$ whenever $\alpha = \{A_i : 0 \leq i < |\pv|\}$ is a partition with $A_i \supseteq A_i^*$ for every $i$.
\end{thm}

\begin{proof}
Combine \cite[Theorem 8.1]{S14} with \cite[Lemma 2.2]{S14}.
\end{proof}

Let $\mathcal{M}$ be as in Lemma \ref{lem:split}. Denote by $\Prt_{\mathcal{M}}$ the set of sequences $\alpha = \{[A_i]_\mu : i \in \N\} \in (\malg_{\mu \res \mathcal{M}})^\N$ where $\alpha = \{A_i : i \in \N\}$ is a $\mathcal{M}$-measurable partition of $X$ (some $A_i$ may be empty). Note that $\Prt_{\mathcal{M}}$ is a Borel subset of $(\malg_{\mu \res \mathcal{M}})^\N$ and is thus a standard Borel space. For notational convenience we will treat each $\alpha \in \Prt_{\mathcal{M}}$ as a $\mathcal{M}$-measurable partition $\alpha = \{A_i : i \in \N\}$ of $X$. This will not cause problems since any two choices for expressing $\alpha$ in this way will only differ on a $\mu$-null set. Moreover, by Lemma \ref{lem:split}.(ii) they will only differ on a $\nu$-null set for $\tau$-almost-every $\nu$.

For a partition $\alpha = \{A_i : i \in \N\}$, define $\theta^\alpha : X \rightarrow \N^G$ as in the previous section: $\theta^\alpha(x)(g) = i \Leftrightarrow g^{-1} \cdot x \in A_i$. Also write $\dist_\mu(\alpha)$ for the probability vector whose $(i+1)^{\text{st}}$ entry is $\mu(A_i)$.

\begin{lem} \label{lem:borelgp}
Let $G \acts (X, \mu)$ and $\mathcal{M}$ be as in Lemma \ref{lem:split}. Let $\xi \subseteq \Borel(X)$ be countable, let $\cF$ be a countably generated $G$-invariant sub-$\sigma$-algebra, and let $\nu \mapsto \pv^\nu$ be a Borel map associating to each $\nu \in \E_G(X)$ a probability vector $\pv^\nu$ satisfying $\rh_{G,\nu}(\xi \given \cF) < \sH(\pv^\nu)$.
\begin{enumerate}
\item[\rm (i)] The set $Z$ of pairs $(\nu, \alpha) \in \E_G(X) \times \Prt_{\mathcal{M}}$ satisfying $\dist_\mu(\alpha) = \pv^\nu$ and $\xi \subseteq \salg_G(\alpha) \vee \cF \mod \Null_\nu$ is Borel.
\item[\rm (ii)] The map $(\nu, \alpha) \in \E_G(X) \times \Prt_{\mathcal{M}} \mapsto \theta^\alpha_*(\nu)$ is Borel.
\item[\rm (iii)] For $\tau$-almost-every $\nu \in \E_G(X)$ the set $\{\theta^\alpha_*(\nu) : (\nu, \alpha) \in Z\}$ is uncountable.
\end{enumerate}
\end{lem}

\begin{proof}
(i). The set of $(\nu, \alpha)$ with $\dist_\mu(\alpha) = \pv^\nu$ is clearly Borel, so it suffices to show that $\{(\nu, \alpha) : \xi \subseteq \salg_G(\alpha) \vee \cF \mod \Null_\nu\}$ is Borel. Let $\{F_i : i \in \N\}$ be a countable $G$-invariant algebra which generates $\cF$. Write $[G \rightharpoonup \N]^{<\infty}$ for the set of functions $f : T \rightarrow \N$ with $T \subseteq G$ finite. For $i \in \N$, $f : T \rightarrow \N$, and $\alpha = \{A_j : j \in \N\} \in \Prt_{\mathcal{M}}$, define $S_{(i,f)}(\alpha) = F_i \cap \bigcap_{t \in T} t \cdot A_{f(t)}$. For finite sets $P \subseteq \N \times [G \rightharpoonup \N]^{<\infty}$ define $S_P(\alpha) = \bigcup_{(i, f) \in P} S_{(i, f)}(\alpha)$. Then $\{S_P(\alpha) : P \subseteq \N \times [G \rightharpoonup \N]^{<\infty} \text{ finite}\}$ is a countable $G$-invariant algebra which generates $\salg_G(\alpha) \vee \cF$. So we have $\xi = \{D_k : k \in \N\}$ is contained in $\salg_G(\alpha) \vee \cF \mod \Null_\nu$ if and only if
$$\forall k \in \N \ \forall n \in \N \ \exists \text{ finite } P \subseteq \N \times [G \rightharpoonup \N]^{<\infty} \quad \nu(D_k \symd S_P(\alpha)) < 1 / n.$$
For fixed $k \in \N$ and finite $P \subseteq \N \times [G \rightharpoonup \N]^{<\infty}$, the map $(\nu, \alpha) \mapsto \nu(D_k \symd S_P(\alpha))$ is Borel by Lemmas \ref{lem:adalg}, \ref{lem:mm}, and \ref{lem:mad}. This establishes (ii).

(ii). The map $(\nu, \alpha) \mapsto \theta^\alpha_*(\nu)$ is Borel if and only if $(\nu, \alpha) \mapsto \theta^\alpha_*(\nu)(B)$ is Borel for every Borel set $B \subseteq \N^G$. The collection of such sets $B$ clearly forms a $\sigma$-algebra, so it suffices to check the case where $B$ is clopen. This case is immediately implied by our argument for (ii).

(iii). Fix $\nu \in \E_G(X)$ with $\mathcal{M} = \Borel(X) \mod \Null_\nu$ and with $\nu$ non-atomic. By assumption $\rh_{G,\nu}(\xi \given \cF) < \sH(\pv^\nu)$. So we can find $0 < r < 1$ with $\rh_{G,\nu}(\xi \given \cF) < r \cdot \sH(\pv^\nu)$. By Theorem \ref{thm:p1strong}, there is a collection $\alpha^* = \{A_i^* : i \in \N\}$ of pairwise disjoint Borel sets with $\mu(A_i^*) = r \cdot p_i^\nu$ for every $i$ and with $\xi \subseteq \salg_G(\alpha) \vee \cF \mod \Null_\nu$ whenever $\alpha = \{A_i : i \in \N\}$ is a partition with $A_i \supseteq A_i^*$. Since $\mathcal{M} = \Borel(X) \mod \Null_\nu$, we may assume that $\alpha^* \subseteq \mathcal{M}$. Since $\nu$ is ergodic and non-atomic, there are a non-identity $g \in G$ and non-null disjoint $\mathcal{M}$-measurable sets $Y_0, Y_1 \subseteq X \setminus \cup \alpha^*$ with $g \cdot Y_0 = Y_1$. By replacing $Y_0, Y_1$ with translates $g^{-k} \cdot Y_0, g^{-k} \cdot Y_1$, $k \in \N$, and by shrinking $Y_0$ if necessary, we may assume that there is $N \in \N$ with $g^{-1} \cdot Y_0 \subseteq A_N^*$. Fix this $N \in \N$. Now fix $\mathcal{M}$-measurable partitions $\beta = \{B_i : i \in \N\}$ of $X \setminus (Y_0 \cup Y_1 \cup (\cup \alpha^*))$ and $\gamma = \{C_i : i \in \N\}$ of $Y_0$ with $\nu(B_i) = \nu(\cup \beta) \cdot p_i^\nu$ and $\nu(C_i) = \nu(Y_0) \cdot p_i^\nu$. Fix a continuous path $\chi^t = \{K_i^t : i \in \N\}$, $0 \leq t \leq 1$, of $\mathcal{M}$-measurable partitions of $Y_0$ such that $\chi^0 = \gamma$, $\chi^1$ is independent with $\gamma$ on $Y_0$ (i.e. $\nu(K_i^1 \cap C_j) / \nu(Y_0) = \nu(K_i^1) \cdot \nu(C_j) / \nu(Y_0)^2$ for every $i, j \in \N$), and $\nu(K_i^t) = \nu(Y_0) \cdot p_i^\nu$ for all $i \in \N$ and $0 \leq t \leq 1$. Set $\alpha^t = \{A_i^t : i \in \N\}$ where
$$A_i^t = A_i^* \cup B_i \cup K_i^t \cup g \cdot C_i.$$
Then $\nu(A_i^t) = p_i^\nu$ and $\xi \subseteq \salg_G(\alpha^t) \vee \cF \mod \Null_\nu$ since $A_i^t \supseteq A_i^*$. So $(\nu, \alpha^t) \in Z$. It suffices to show that $\nu(A_N^t \cap g^{-1} \cdot A_N^t)$ takes uncountably many values as $t$ varies. Notice that the measure of the sets $(A_N^t \setminus Y_0) \cap g^{-1} \cdot (A_N^t \setminus Y_0)$ and
$$(A_N^t \setminus Y_0) \cap g^{-1} (A_N^t \cap Y_0) = (A_N^t \setminus Y_0) \cap g^{-1} \cdot K_N^t = A_N^* \cap g^{-1} \cdot K_N^t = g^{-1} \cdot K_N^t$$
do not depend on $t$ and that these sets partition $(A_N^t \setminus Y_0) \cap g^{-1} \cdot A_N^t$. The remaining portion of $A_N^t \cap g^{-1} \cdot A_N^t$ has measure $\nu(A_N^t \cap Y_0 \cap g^{-1} \cdot A_N^t) = \nu(K_N^t \cap C_N)$ which varies continuously from $p_N^\nu$ to $(p_N^\nu)^2$. We cannot have $p_N^\nu = (p_N^\nu)^2$ as otherwise $p_N^\nu = 0$ and $Y_0 \subseteq g \cdot A_N^*$ is $\nu$-null or $p_N^\nu = 1$ and $\sH(\pv^\nu) = 0$, both of which are contradictions.
\end{proof}

Now we are ready for the main theorem. Note we obtain the weaker Theorem \ref{intro:nekrieger} by choosing a countable collection $\xi \subseteq \Borel(X)$ with $\salg(\xi) = \Borel(X)$.

\begin{thm} \label{thm:nekrieger}
Let $G \acts (X, \mu)$ be an aperiodic {\pmp} action, let $\xi \subseteq \Borel(X)$ be countable, let $\cF$ be a countably generated $G$-invariant sub-$\sigma$-algebra, and let $\mu = \int_{\E_G(X)} \nu \ d \tau(\nu)$ be the ergodic decomposition of $\mu$. If $\nu \mapsto \pv^\nu = \{p_i^\nu : i \in \N\}$ is a Borel map associating to each $\nu \in \E_G(X)$ a probability vector $\pv^\nu$ satisfying $\rh_{G,\nu}(\xi \given \cF) < \sH(\pv^\nu)$, then there is a Borel partition $\alpha = \{A_i : i \in \N\}$ of $X$ such that $\xi \subseteq \salg_G(\alpha) \vee \cF \mod \Null_\mu$ and such that $\nu(A_i) = p_i^\nu$ for every $i \in \N$ and $\tau$-almost-every $\nu$.
\end{thm}

\begin{proof}
Let $\mathcal{M}$ be as given by Lemma \ref{lem:split}. Let $Z \subseteq \E_G(X) \times \Prt_{\mathcal{M}}$ be the set of pairs $(\nu, \alpha)$ such that $\xi \subseteq \salg_G(\alpha) \vee \cF \mod \Null_\nu$ and $\dist_\mu(\alpha) = \pv^\nu$. Note that for $\tau$-almost-every $\nu \in \E_G(X)$ and every $\alpha \in \Prt_{\mathcal{M}}$ we have $\dist_\nu(\alpha) = \dist_\mu(\alpha)$ by Lemma \ref{lem:split}.(ii). Lemma \ref{lem:borelgp} shows that $Z$ and the function $(\nu, \alpha) \mapsto \theta^\alpha_*(\nu)$ satisfy the assumption of Proposition \ref{prop:select}. So that proposition gives a $\tau$-conull Borel set $E \subseteq \E_G(X)$ and a Borel function $\phi : E \rightarrow \Prt_{\mathcal{M}}$ whose graph is contained in $Z$ and with the map $\nu \in E \mapsto \theta^{\phi(\nu)}_*(\nu)$ injective.

By Lemmas \ref{lem:fibers} and \ref{lem:mad} there is a Borel partition $\alpha$ of $X$ such that $\phi(\nu) = \alpha \mod \Null_\nu$ for every $\nu \in E$. Then $\dist_\nu(\alpha) = \pv^\nu$ for $\tau$-almost-every $\nu \in E$. Also $\xi \subseteq \salg_G(\alpha) \vee \cF \mod \Null_\nu$ for every $\nu \in E$, so Lemma \ref{lem:fibers} implies that $\xi \subseteq \sinv_G \vee \salg_G(\alpha) \vee \cF \mod \Null_\mu$. On the other hand, $\nu \in E \mapsto \theta^\alpha_*(\nu)$ is injective by construction and thus $\sinv_G \subseteq \salg_G(\alpha) \mod \Null_\mu$ by Lemma \ref{lem:sinv}. Therefore $\xi \subseteq \salg_G(\alpha) \vee \cF \mod \Null_\mu$.
\end{proof}

As a consequence we obtain an ergodic decomposition formula for Rokhlin entropy. First we must note a simple but technical fact.

\begin{lem} \label{lem:entan}
Let $G \acts (X, \mu)$ be a {\pmp} aperiodic action, let $\xi \subseteq \Borel(X)$ be countable, and let $\cF$ be a countably generated $G$-invariant sub-$\sigma$-algebra. Let $\mu = \int_{\E_G(X)} \nu \ d \tau(\nu)$ be the ergodic decomposition of $\mu$. Then there is a Borel $\tau$-conull set $E \subseteq \E_G(X)$ such that the map $\nu \in E \mapsto \rh_{G,\nu}(\xi \given \cF)$ is Borel.
\end{lem}

\begin{proof}
Let $\mathcal{M}$ be as in Lemma \ref{lem:split}. By the proof of Lemma \ref{lem:borelgp}.(i) the set $W$ of $(\nu, \alpha) \in \E_G(X) \times \Prt_{\mathcal{M}}$ satisfying $\xi \subseteq \salg_G(\alpha) \vee \cF \mod \Null_\nu$ is Borel. So for every $r \geq 0$ the set
\begin{align*}
\{\nu \in \E_G(X) : \rh_{G,\nu}(\xi \given \cF) \leq r\} & = \{\nu \in \E_G(X) : \exists \alpha \in \Prt_{\mathcal{M}} \ (\nu, \alpha) \in W \wedge \sH_\nu(\alpha) \leq r\}
\end{align*}
is analytic. Thus the map $\nu \mapsto \rh_{G,\nu}(\xi \given \cF)$ is $\tau$-measurable \cite[Thm. 21.10]{K95}. It follows that there is a Borel conull set $E \subseteq \E_G(X)$ on which this map is Borel (for details see the proof of Proposition \ref{prop:select}).
\end{proof}

\begin{cor} \label{cor:ergavg}
Let $G \acts (X, \mu)$ be a {\pmp} action, let $\xi \subseteq \Borel(X)$ be countable, and let $\cF$ be a countably generated $G$-invariant sub-$\sigma$-algebra. If $\mu = \int_{\E_G(X)} \nu \ d \tau(\nu)$ is the ergodic decomposition of $\mu$ then
$$\rh_{G,\mu}(\xi \given \cF) = \int_{\E_G(X)} \rh_{G, \nu}(\xi \given \cF) \ d \tau(\nu).$$
\end{cor}

\begin{proof}
Let $X_\infty \subseteq X$ be the set of points having an infinite $G$-orbit. Also let $E_\infty \subseteq \E_G(X)$ be the set of $\nu$ for which $G \acts (X, \nu)$ is aperiodic. Note that the probability measure $\mu_\infty = \frac{1}{\tau(E_\infty)} \cdot \int_{E_\infty} \nu \ d \tau(\nu)$ is $G$-invariant and supported on $X_\infty$. Fix $\epsilon > 0$. By Lemma \ref{lem:entan} we may fix a Borel map $\nu \in E_\infty \mapsto \pv^\nu$ satisfying $\sH(\pv^\nu) = \rh_{G,\nu}(\xi \given \cF) + \epsilon$ for $\tau$-almost-every $\nu$. By Theorem \ref{thm:nekrieger} there is a partition $\alpha_0$ of $X_\infty$ satisfying $\sH_\nu(\alpha_0) = \rh_{G,\nu}(\xi \given \cF) + \epsilon$ for almost-every $\nu \in E_\infty$ and $\xi \subseteq \salg_G(\alpha_0) \vee \cF \mod \Null_{\mu_\infty}$. Let $X_*$ be the set of points $x \in X \setminus X_\infty$ such that the restriction $\xi \res G \cdot x$ is not a subset of $\cF \res G \cdot x$. Since all orbits in $X_*$ are finite, there is a Borel set $B \subseteq X_*$ which meets every orbit in $X_*$ precisely once. Set $\alpha = \alpha_0 \cup \{B, X \setminus (X_\infty \cup B)\}$. Note that $\alpha$ is a partition of $X$ and that $\sH_\nu(\alpha \given \cF) = \rh_{G,\nu}(\xi \given \cF)$ for all $\nu \in \E_G(X) \setminus E_\infty$. By our choice of $B$ and $\alpha_0$, we have that $\xi \subseteq \salg_G(\alpha) \vee \sinv_G \vee \cF \mod \Null_\mu$. So by definition of Rokhlin entropy and Lemma \ref{lem:break} we have
\begin{align*}
\rh_{G,\mu}(\xi \given \cF) \leq \sH_\mu(\alpha \given \sinv_G \vee \cF) & = \int_{\nu \in \E_G(X)} \sH_\nu(\alpha \given \cF) \ d \tau(\nu)\\
& \leq \epsilon + \int_{\nu \in \E_G(X)} \rh_{G, \nu}(\xi \given \cF) \ d \tau(\nu).
\end{align*}
By letting $\epsilon$ tend to $0$ we obtain one inequality.

For the reverse inequality, suppose that $\alpha$ is a countable partition satisfying $\xi \subseteq \salg_G(\alpha) \vee \sinv_G \vee \cF \mod \Null_\mu$. Since $\xi$ is countable, for $\tau$-almost-every $\nu \in \E_G(X)$ we have $\xi \subseteq \salg_G(\alpha) \vee \cF \mod \Null_\nu$. It follows that $\tau$-almost-always $\sH_\nu(\alpha \given \cF) \geq \rh_{G,\nu}(\xi \given \cF)$. Therefore applying Lemma \ref{lem:break} we get
$$\sH_\mu(\alpha \given \sinv_G \vee \cF) = \int_{\nu \in \E_G(X)} \sH_\nu(\alpha \given \cF) \ d \tau(\nu) \geq \int_{\nu \in \E_G(X)} \rh_{G,\nu}(\xi \given \cF) \ d \tau(\nu).$$
Now take the infimum over all such $\alpha$ to obtain $\rh_{G,\mu}(\xi \given \cF)$ on the left-hand side.
\end{proof}

Now we can verify the countable sub-additivity property of Rokhlin entropy for non-ergodic actions. At the moment, countable sub-additivity is arguably the most useful property for studying Rokhlin entropy. At first glance, this property may seem like an immediate consequence of the definitions, but this is not so. For example, this sub-additive property implies that if $\salg_G(\alpha \vee \beta) = \Borel(X)$ then $\rh_G(X, \mu) \leq \sH(\alpha) + \sH(\beta \given \salg_G(\alpha))$. Its proof relies critically upon Theorem \ref{thm:nekrieger} (or Theorem \ref{intro:krieger} in the ergodic case).

\begin{cor} \label{cor:add2}
Let $G \acts (X, \mu)$ be a {\pmp} action, let $\cF$ be a $G$-invariant sub-$\sigma$-algebra, and let $\xi \subseteq \Borel(X)$. If $(\Sigma_n)_{n \in \N}$ is an increasing sequence of $G$-invariant sub-$\sigma$-algebras with $\xi \subseteq \bigvee_{n \in \N} \Sigma_n \vee \cF$ then
$$\rh_{G,\mu}(\xi \given \cF) \leq \rh_{G,\mu}(\Sigma_1 \given \cF) + \sum_{n = 2}^\infty \rh_{G,\mu}(\Sigma_n \given \Sigma_{n-1} \vee \cF).$$
\end{cor}

\begin{proof}
Let $\xi' \subseteq \Borel(X)$ be countable with $\salg(\xi) = \salg(\xi') \mod \Null_\mu$. Also fix countably generated $\sigma$-algebras $\cF'$ and $\Sigma_n'$ with $\cF' = \cF \mod \Null_\mu$ and $\Sigma_n' = \Sigma_n \mod \Null_\mu$. Clearly $\rh_{G,\mu}(\xi' \given \cF') = \rh_{G, \mu}(\xi \given \cF)$ and $\rh_{G, \mu}(\Sigma_n' \given \Sigma_{n-1}' \vee \cF') = \rh_{G,\mu}(\Sigma_n \given \Sigma_{n-1} \vee \cF)$. It was recorded in \cite[Cor. 2.5]{S14b} that for ergodic $\nu \in \E_G(X)$ we have
$$\rh_{G,\nu}(\xi' \given \cF') \leq \rh_{G,\nu}(\Sigma_1' \given \cF') + \sum_{n = 2}^\infty \rh_{G,\nu}(\Sigma_n' \given \Sigma_{n-1}' \vee \cF').$$
Now integrate over $\nu \in \E_G(X)$ and apply Corollary \ref{cor:ergavg}.
\end{proof}

\begin{cor} \label{cor:add1}
Let $G \acts (X, \mu)$ be a {\pmp} action, and let $\cF$ be a $G$-invariant sub-$\sigma$-algebra. If $G \acts (Y, \nu)$ is a factor of $(X, \mu)$ and $\Sigma$ is the associated $G$-invariant sub-$\sigma$-algebra, then
$$\rh_G(X, \mu \given \cF) \leq \rh_G(Y, \nu) + \rh_G(X, \mu \given \cF \vee \Sigma).$$
\end{cor}

\begin{proof}
Apply Corollary \ref{cor:add2} using $\xi = \Borel(X)$ and note that $\rh_{G,\mu}(\Sigma \given \cF) \leq \rh_G(Y, \nu)$.
\end{proof}

Using Theorem \ref{thm:nekrieger} and the ergodic decomposition formula, we obtain a simplified expression for Rokhlin entropy in the case of aperiodic actions. From the original definition, the expressions $\sH(\alpha \given \cF \vee \sinv_G)$ and $\xi \subseteq \salg_G(\alpha) \vee \cF \vee \sinv_G$ are replaced with $\sH(\alpha)$ and $\xi \subseteq \salg_G(\alpha) \vee \cF$ below.

\begin{cor} \label{cor:simplerok}
Let $G \acts (X, \mu)$ be a {\pmp} action, let $\xi \subseteq \Borel(X)$, and let $\cF$ be a $G$-invariant sub-$\sigma$-algebra. If $G \acts (X, \mu)$ is aperiodic then
$$\rh_{G,\mu}(\xi \given \cF) = \inf \Big\{ \sH(\alpha) : \alpha \text{ a countable partition with } \xi \subseteq \salg_G(\alpha) \vee \cF\Big\}.$$
\end{cor}

\begin{proof}
It is immediate from the definitions that the infimum on the right is greater than or equal to $\rh_{G,\mu}(\xi \given \cF)$, so it suffices to check the reverse inequality. The argument we present essentially comes from \cite{Roh67}. Fix $\epsilon > 0$ and pick $n \in \N$ with $1/n < \epsilon / 2$. Let $\mu = \int_{\E_G(X)} \nu \ d \tau(\nu)$ be the ergodic decomposition of $\mu$. For $\nu \in \E_G(X)$ define $t(\nu) = \frac{j}{n}$ where $j \in \N$ is least with $\rh_{G,\nu}(\xi \given \cF) < \frac{j}{n}$. By Lemma \ref{lem:entan} there is a Borel $\tau$-conull set $E \subseteq \E_G(X)$ on which the function $t$ is Borel. Set $T = \int t(\nu) \ d \tau(\nu)$ and note that $T < \rh_{G,\mu}(\xi \given \cF) + \epsilon / 2$ by Corollary \ref{cor:ergavg}. Now fix $k \in \N$ large enough that
$$\sH \left(1 - \frac{T}{\log(k)}, \frac{T}{\log(k)} \right) < \frac{\epsilon}{2},$$
and for $\nu \in \E_G(X)$ set
$$\pv^\nu = \left(1 - \frac{t(\nu)}{\log(k)}, \frac{t(\nu)}{k \log(k)}, \frac{t(\nu)}{k \log(k)}, \cdots, \frac{t(\nu)}{k \log(k)} \right).$$
Then
\begin{align*}
\sH(\pv^\nu) & = \sH \left(1 - \frac{t(\nu)}{\log(k)}, \frac{t(\nu)}{\log(k)} \right) + \frac{t(\nu)}{\log(k)} \cdot \sH \left(\frac{1}{k}, \frac{1}{k}, \cdots, \frac{1}{k} \right)\\
 & \geq 0 + t(\nu) > \rh_{G,\nu}(\xi \given \cF).
\end{align*}
We apply Theorem \ref{thm:nekrieger} to get a Borel partition $\alpha = \{A_i : i \in \N\}$ satisfying $\xi \subseteq \salg_G(\alpha) \vee \cF$ and $\nu(A_i) = p^\nu_i$ for every $i \in \N$ and $\tau$-almost-every $\nu$. Finally, observe that $\sH(\alpha)$ is equal to
\begin{align*}
\sH \left(1 - \frac{T}{\log(k)}, \frac{T}{\log(k)} \right) + \frac{T}{\log(k)} \cdot \sH \left(\frac{1}{k}, \frac{1}{k}, \ldots, \frac{1}{k} \right) < \frac{\epsilon}{2} + T < \rh_{G,\mu}(\xi \given \cF) + \epsilon.
\end{align*}
Letting $\epsilon$ tend to $0$ shows that the infimum is at most $\rh_{G,\mu}(\xi \given \cF)$.
\end{proof}

\section{Semi-continuity properties} \label{sec:cont}

In this section we establish some continuity and upper-semicontinuity results for Rokhlin entropy. Recall that a real-valued function $f$ on a topological space $X$ is called \emph{upper-semicontinuous} if for every $x \in X$ and $\epsilon > 0$ there is an open set $U$ containing $x$ with $f(y) < f(x) + \epsilon$ for all $y \in U$. When $X$ is first countable, this is equivalent to saying that $f(x) \geq \limsup f(x_n)$ whenever $(x_n)$ is a sequence converging to $x$.

For a probability space $(X, \mu)$, we will work with the space $\HPrt(\mu)$ of countable Borel partitions having finite Shannon entropy. If $\cF \subseteq \Borel(X)$ is a sub-$\sigma$-algebra, we write $\HPrt(\cF, \mu)$ for the set of $\cF$-measurable partitions in $\HPrt(\mu)$. The set $\HPrt(\mu)$ becomes a complete separable metric space when equipped with the \emph{Rokhlin metric} $\dR_\mu$ defined by $\dR_\mu(\alpha, \beta) = \sH_\mu(\alpha \given \beta) + \sH_\mu(\beta \given \alpha)$ \cite[Fact 1.7.15]{Do11}. We record some useful inequalities for the Rokhlin metric. Below, if $G$ acts on $(X, \mu)$, $\alpha$ is a partition of $X$, and $T \subseteq G$ is finite, then we let $\alpha^T$ denote the join $\bigvee_{t \in T} t \cdot \alpha$.

\begin{lem} \label{lem:app}
Let $G \acts (X, \mu)$ be a {\pmp} action, let $\cF$ be a $G$-invariant sub-$\sigma$-algebra, and let $\alpha, \beta, \xi \in \HPrt(\mu)$. Then:
\begin{enumerate} 
\item[\rm (i)] $\dR_\mu(\beta^T, \xi^T) \leq |T| \cdot \dR_\mu(\beta, \xi)$ for every finite $T \subseteq G$;
\item[\rm (ii)] $|\sH(\beta \given \cF) - \sH(\xi \given \cF)| \leq \dR_\mu(\beta, \xi)$;
\item[\rm (iii)] $|\sH(\alpha \given \beta \vee \cF) - \sH(\alpha \given \xi \vee \cF)| \leq 2 \cdot \dR_\mu(\beta, \xi)$.
\end{enumerate}
\end{lem}

\begin{proof}
This is a simple exercise. Alternatively, see the appendix to \cite{S14b}.
\end{proof}

We begin with a few simple cases in which Rokhlin entropy is actually continuous, not just semicontinuous. Below for a sub-$\sigma$-algebra $\cF$ of $(X, \mu)$ and partitions $\alpha$ and $\beta$ with $\sH(\alpha \given \cF), \sH(\beta \given \cF) < \infty$, we define
$$\dR_\mu(\alpha, \beta \given \cF) = \sH(\alpha \given \beta \vee \cF) + \sH(\beta \given \alpha \vee \cF).$$
Note this quantity is bounded above by $\dR_\mu(\alpha, \beta)$ when $\alpha, \beta \in \HPrt(\mu)$.

\begin{lem} \label{lem:cont}
Let $G \acts (X, \mu)$ be a {\pmp} action, and let $\cF$ be a $G$-invariant sub-$\sigma$-algebra. Let $\gamma$, $\zeta$, $\cP$, and $\cQ$ be partitions such that $\sH(\gamma \given \cF), \sH(\zeta \given \cF), \sH(\cP \given \cF), \sH(\cQ \given \cF) < \infty$. 
\begin{enumerate}
\item[\rm (i)] $|\rh_{G,\mu}(\xi \given \salg_G(\gamma) \vee \cF) - \rh_{G,\mu}(\xi \given \salg_G(\zeta) \vee \cF)| \leq \dR_{\mu}(\gamma, \zeta \given \cF)$ for every collection $\xi \subseteq \Borel(X)$.
\item[\rm (ii)] $|\rh_G(X, \mu \given \salg_G(\gamma) \vee \cF) - \rh_G(X, \mu \given \salg_G(\zeta) \vee \cF)| \leq \dR_{\mu}(\gamma, \zeta \given \cF)$.
\item[\rm (iii)] $|\rh_{G,\mu}(\cP \given \cF) - \rh_{G,\mu}(\cQ \given \cF)| \leq \dR_\mu(\cP, \cQ \given \cF)$.
\end{enumerate}
\end{lem}

\begin{proof}
(i). By sub-additivity (Corollary \ref{cor:add2})
\begin{align*}
\rh_{G,\mu}(\xi \given \salg_G(\gamma) \vee \cF) & \leq \rh_{G,\mu}(\zeta \given \salg_G(\gamma) \vee \cF) + \rh_{G,\mu}(\xi \given \salg_G(\zeta) \vee \cF)\\
 & \leq \sH(\zeta \given \gamma \vee \cF) + \rh_{G,\mu}(\xi \given \salg_G(\zeta) \vee \cF).
\end{align*}
By symmetry a similar inequality holds with $\gamma$ and $\zeta$ reversed.

(ii). Use $\xi = \Borel(X)$ and apply (i).

(iii). By sub-additivity we have
\begin{align*}
\rh_{G,\mu}(\cP \given \cF) & \leq \rh_{G,\mu}(\cQ \given \cF) + \rh_{G,\mu}(\cP \given \salg_G(\cQ) \vee \cF)\\
 & \leq \rh_{G,\mu}(\cQ \given \cF) + \sH(\cP \given \cQ \vee \cF).
\end{align*}
By symmetry a similar inequality holds with $\cP$ and $\cQ$ reversed.
\end{proof}

Now we present a general but rather technical formula for Rokhlin entropy. This lemma will be used both in this section and the next in order to study the nature of Rokhlin entropy. First we need some additional notation. We write $\beta \leq \alpha$ if the partition $\beta$ is coarser than the partition $\alpha$. Also, we say a collection of partitions $\mathcal{C}$ is \emph{c-dense} (coarsely-dense) in $\HPrt(\cF, \mu)$ if for every $\psi \in \HPrt(\cF, \mu)$ and $\epsilon > 0$ there is $\gamma \in \mathcal{C}$ and a coarsening $\psi' \leq \gamma$ with $\dR_\mu(\psi', \psi) < \epsilon$. For an action $G \acts (X, \mu)$ we write $\sfinv_G$ for the $\sigma$-algebra generated by the Borel $G$-invariant sets consisting only of points having finite $G$-orbits. In other words, writing $X_{< \infty} = \{x \in X : |G \cdot x| < \infty\}$, the $\sigma$-algebra $\sfinv_G$ consists precisely of those Borel sets $A \subseteq X$ such that $A$ is $G$-invariant and either $A$ or $X \setminus A$ is a subset $X_{< \infty}$.

\begin{lem} \label{lem:outkolm}
Let $G \acts (X, \mu)$ be a {\pmp} action. Let $\cF$ be a $G$-invariant sub-$\sigma$-algebra, and let $\cP$ be a countable partition. Let $\mathcal{A} \subseteq \HPrt(\mu)$ be a collection of partitions which is c-dense in $\HPrt(\mu)$, and let $\mathcal{C} \subseteq \HPrt(\cF \vee \sfinv_G, \mu)$ be a collection of partitions which is c-dense in $\HPrt(\cF \vee \sfinv_G, \mu)$. If $\sH(\cP) < \infty$ then $\rh_{G,\mu}(\cP \given \cF)$ is equal to
$$\lim_{\epsilon \rightarrow 0} \inf_{\substack{\alpha \in \mathcal{A}\\\gamma \in \mathcal{C}}} \inf_{\substack{T \subseteq G\\T \text{finite}}} \inf \Big\{ \sH(\beta \given \chi^T \vee \gamma) : \beta, \chi \leq \alpha, \ \sH(\chi) + \sH(\cP \given \beta^T \vee \chi^T \vee \gamma) < \epsilon \Big\}.$$
In fact, for every $\epsilon > 0$ the triple infimum above is bounded between $\rh_{G,\mu}(\cP \given \cF) - \epsilon$ and $\rh_{G, \mu}(\cP \given \cF)$.
\end{lem}

\begin{proof}
It suffices to prove the second claim. Fix $\epsilon > 0$. Consider $\alpha \in \mathcal{A}$, $\gamma \in \mathcal{C} \subseteq \HPrt(\cF \vee \sfinv_G, \mu)$, finite $T \subseteq G$, and $\beta, \chi \leq \alpha$ with $\sH(\chi) + \sH(\cP \given \beta^T \vee \chi^T \vee \gamma) < \epsilon$. By sub-additivity of Rokhlin entropy we have
\begin{align*}
\rh_{G,\mu}(\cP & \given \cF)\\
 & \leq \rh_{G,\mu}(\chi \given \cF) + \rh_{G,\mu}(\beta \given \salg_G(\chi) \vee \cF) + \rh_{G,\mu}(\cP \given \salg_G(\beta \vee \chi) \vee \cF)\\
 & \leq \sH(\chi) + \sH(\beta \given \chi^T \vee \gamma) + \sH(\cP \given \beta^T \vee \chi^T \vee \gamma)\\
 & < \sH(\beta \given \chi^T \vee \gamma) + \epsilon.
\end{align*}
This establishes the first inequality.

Now we consider the second inequality. Fix $\epsilon > 0$. Note that $\rh_{G, \mu}(\cP \given \cF) \leq \sH(\cP) < \infty$. Fix $\delta > 0$. By the definition of Rokhlin entropy, there is a countable partition $\xi$ satisfying
\begin{equation}\label{eqn:outkolm1}
\cP \subseteq \salg_G(\xi) \vee \cF \vee \sinv_G \quad \text{and} \quad \sH(\xi \given \cF \vee \sinv_G) < \rh_{G, \mu}(\cP \given \cF) + \delta / 3.
\end{equation}
We claim we can further assume that $\sH(\xi) < \infty$, since $\rh_{G,\mu}(\cP \given \cF) < \infty$. Indeed, on the subset of $X$ consisting of points having finite $G$-orbit we can choose $\xi$ to consist of two sets, one of which meets every finite $G$-orbit in precisely one point, and on the subset of $X$ consisting of points with infinite $G$-orbit we can appeal to Corollary \ref{cor:simplerok}. By similar reasoning, Corollary \ref{cor:invzero} implies there is a partition $\omega$ with $\sH(\omega) < \epsilon / 4$ and $\sinv_G \subseteq \sfinv_G \vee \salg_G(\omega)$. Note that our assumptions then imply $\{\omega^T \vee \gamma : T \subseteq G \text{ finite}, \ \gamma \in \mathcal{C}\}$ is c-dense in $\HPrt(\cF \vee \sinv_G, \mu)$. By (\ref{eqn:outkolm1}) we can find finite $T \subseteq G$ and $\gamma \in \mathcal{C}$ with
$$\sH(\cP \given \xi^T \vee \omega^T \vee \gamma) < \epsilon / 6 \quad \text{and} \quad \sH(\xi \given \omega^T \vee \gamma) < \rh_{G, \mu}(\cP \given \cF) + \delta / 3.$$
Next, since $\mathcal{A}$ is c-dense in $\HPrt(\mu)$, we can find $\alpha \in \mathcal{A}$ and partitions $\beta, \chi \leq \alpha$ with
$$\dR_\mu(\beta, \xi) < \min \left( \frac{\epsilon}{12 |T|}, \frac{\delta}{3} \right) \quad \text{and} \quad \dR_\mu(\chi, \omega) < \min \left( \frac{\epsilon}{12 |T|}, \frac{\delta}{6 |T|} \right).$$
Then
\begin{align*}
\sH( & \chi) + \sH(\cP \given \beta^T \vee \chi^T \vee \gamma)\\
 & \leq \sH(\omega) + \dR_\mu(\chi, \omega) + \sH(\cP \given \xi^T \vee \omega^T \vee \gamma) + 2 |T| \cdot \dR_\mu(\beta, \xi) + 2 |T| \cdot \dR_\mu(\chi, \omega)\\
 & < \epsilon.
\end{align*}
Furthermore,
\begin{align*}
\sH(\beta \given \chi^T \vee \gamma) & \leq \sH(\xi \given \omega^T \vee \gamma) + \dR_\mu(\beta, \xi) + 2 |T| \dR_\mu(\chi, \omega)\\
 & < \rh_{G,\mu}(\cP \given \cF) + \delta.
\end{align*}
Therefore
$$\inf_{\substack{\alpha \in \mathcal{A}\\\gamma \in \mathcal{C}}} \inf_{\substack{T \subseteq G\\T \text{ finite}}} \inf \Big\{ \sH(\beta \given \chi^T \vee \gamma) : \beta, \chi \leq \alpha, \ \sH(\chi) + \sH(\cP \given \beta^T \vee \chi^T \vee \gamma) < \epsilon \Big\}$$
is less than $\rh_{G,\mu}(\cP \given \cF) + \delta$. Now let $\delta$ tend to $0$.
\end{proof}

For the remainder of this section we will study upper-semicontinuity of Rokhlin entropy in three settings: as a function of the invariant measure, as a function of the partition, and as a function of the action. We will obtain our strongest upper-semicontinuity results when $G$ is finitely generated. Unfortunately, when $G$ is not finitely generated Rokhlin entropy is not upper-semicontinuous in general, as the following example illustrates.

\begin{example}
Consider a non-finitely generated abelian group $G$. Let $(\Gamma_k)_{k \in \N}$ be an increasing sequence of finitely generated subgroups which union to $G$. Write $2$ for the set $\{0, 1\}$ and let $u_2$ be the uniform probability measure on $2$. We will consider $\rh_G(2^G, \mu)$ as $\mu \in \M_G(2^G)$ varies. For a subgroup $\Gamma \leq G$ we naturally identify $2^{G / \Gamma}$ with the set of $x \in 2^G$ which are constant on each $\Gamma$-coset. Through this identification, we view the product measure $u_2^{G / \Gamma} \in \M_G(2^{G / \Gamma})$ as a measure on $2^G$. It is not difficult to see that $u_2^{G / \Gamma_k}$ converges as $k \rightarrow \infty$ to $u_2^{G / G}$ (which is supported on the two constant functions). Clearly $\rh_G(2^G, u_2^{G / G}) = 0$. However, for $k \in \N$ we can view the action $G \acts (2^G, u_2^{G / \Gamma_k})$ as a free action of $G / \Gamma_k$, and since this action is isomorphic to the Bernoulli action $G / \Gamma_k \acts (2^{G / \Gamma_k}, u_2^{G / \Gamma_k})$ of the infinite abelian group $G / \Gamma_k$, we obtain $\rh_G(2^G, u_2^{G / \Gamma_k}) = \log(2)$. Thus Rokhlin entropy is not an upper-semicontinuous function on $\M_G(2^G)$.
\end{example}

With a bit more effort, one can use the above construction to obtain the same conclusion whenever $G$ is non-finitely generated and amenable. We believe this failure of upper-semicontinuity occurs precisely when $G$ is not finitely generated, but we cannot yet prove this since computable lower bounds to Rokhlin entropy for non-sofic actions do not currently exist.

Our stronger upper-semicontinuity results for finitely generated groups will depend upon the following lemma. Throughout this section and the next, for any set $L$ we let $G$ act on $L^G$ by left-shifts: for $x \in L^G$ and $g, h \in G$ we have $(g \cdot x)(h) = x(g^{-1} h)$.

\begin{lem} \label{lem:fcl}
Let $G$ be a finitely generated infinite group, let $L$ be finite, and let $P \subseteq L^G$ be a finite $G$-invariant set. For every $\epsilon > 0$ and open set $U \supseteq P$, there is a clopen set $V$ such that $U \supseteq V \supseteq P$ and $\rh_{G,\mu}(V) < \epsilon$ for all $\mu \in \M_G(L^G)$.
\end{lem}

\begin{proof}
Fix $r \in \N$ with $\sH(r^{-1}, 1 - r^{-1}) < \epsilon$. Fix a finite generating set $S$ for $G$, and for each $n$ let $B_n \subseteq G$ be the corresponding ball of radius $n$. Since the set $\{x \in L^G : |G \cdot x| < r\}$ is finite and $P$ is $G$-invariant, we can find a clopen set $W \supseteq P$ such that $B_r^{-1} B_r \cdot W \subseteq U$ and $W \cap \{x \in L^G : |G \cdot x| < r\} \subseteq P$.

If some $x \in L^G$ satisfies $|B_{n+1} \cdot x| = |B_n \cdot x|$, then $|G \cdot x| = |B_n \cdot x|$. So we must have $|B_r \cdot x| \geq r$ for all $x \in W \setminus P$. In other words, given $x \in W \setminus P$ there are $s_x(0), s_x(1), \ldots, s_x(r-1) \in B_r$ with $s_x(k) \cdot x \neq s_x(m) \cdot x$ for all $k \neq m$. It follows that there is countable cover $\{U_i : i \geq 1\}$ of $W \setminus P$, with each $U_i$ a clopen subset of $W \setminus P$, and a collection of functions $s_i : \{0, \ldots, r-1\} \rightarrow B_r$ such that $s_i(k) \cdot U_i \cap s_i(m) \cdot U_i = \varnothing$ for all $i$ and all $0 \leq k \neq m < r$.

Now inductively define clopen sets $Y_i$ by setting $Y_0 = \varnothing$ and for $i \geq 1$
$$Y_i = Y_{i-1} \cup (U_i \setminus B_r^{-1} B_r \cdot Y_{i-1}).$$
Set $Y_\infty = \bigcup_i Y_i$. Note that each $Y_i$ is clopen and hence $Y_\infty$ is open. Also note that $Y_\infty \subseteq W$ since each $U_i \subseteq W$.

We claim that $Y_\infty \cup P$ is closed. Fix a point $x \in L^G \setminus (Y_\infty \cup P)$. We will find an open neighborhood of $x$ which is disjoint from $Y_\infty \cup P$. If $x \not\in W$ then $L^G \setminus W$ is the desired open neighborhood. Now suppose $x \in W$. Then $x \in W \setminus P$ so there is an $i \geq 1$ with $x \in U_i$. We must have $x \not\in Y_i \subseteq Y_\infty$, and thus the construction implies that $x \in B_r^{-1} B_r \cdot Y_{i-1} \setminus Y_{i-1}$. This is an open set which is disjoint from $Y_\infty \cup P$ (recall $P$ is $G$-invariant and each $U_i$ is disjoint from $P$). This proves the claim.

We set $V = (B_r^{-1} B_r \cdot Y_\infty) \cup P$ and claim that it has the desired properties. Its immediate that $P \subseteq V$ and $V \subseteq B_r^{-1} B_r \cdot W \subseteq U$. Also, since $P$ is $G$-invariant, $V = B_r^{-1} B_r \cdot (Y_\infty \cup P)$ is closed by the previous paragraph. We claim that $V$ is open (hence clopen). As a first step, we argue that $W \subseteq V$. Fix $w \in W$. If $w \in P$ then we are done. Otherwise there is $i \geq 1$ with $w \in U_i$. From the construction, we see that either $w \in B_r^{-1} B_r \cdot Y_{i-1} \subseteq V$ or else $w \in Y_i \subseteq V$. Thus $W \subseteq V$. As $P \subseteq W \subseteq V$, we can write $V$ as $V = (B_r^{-1} B_r \cdot Y_\infty) \cup W$, which shows that $V$ is open.

Finally, since $V = B_r^{-1} B_r \cdot Y_\infty \cup P$ and $P \in \sinv_G$ is $G$-invariant, we have
$$\forall \mu \in \M_G(L^G) \quad \rh_{G,\mu}(V) = \rh_{G,\mu}(B_r^{-1} B_r \cdot Y_\infty) = \rh_{G,\mu}(Y_\infty) \leq \sH_\mu(Y_\infty).$$
By our choice of $r$ it suffices to show that $\mu(Y_\infty) < r^{-1}$ for all $\mu \in \M_G(L^G)$. For $0 \leq k < r$ define $\theta(k) : Y_\infty \rightarrow L^G$ as follows: for $y \in Y_\infty$ choose $i$ least with $y \in Y_i$ and set $\theta(k)(y) = s_i(k) \cdot y$. For fixed $i$, the sets $\theta(k)(Y_i) \subseteq s_i(k) \cdot U_i$, $0 \leq k < r$, are pairwise disjoint. Also, $\bigcup_{k = 0}^{r-1} \theta(k)(Y_i) \subseteq B_r \cdot Y_i$ and $B_r \cdot Y_i \cap B_r \cdot Y_j = \varnothing$ for $i \neq j$. Thus the maps $\theta(k) : Y_\infty \rightarrow L^G$ are injective, measure-preserving, and have pairwise-disjoint images. Hence $\mu(Y_\infty) \leq 1 / r$ as required.
\end{proof}

Now we establish upper-semicontinuity on certain spaces of $G$-invariant measures. Recall that $\M_G^{\text{aper}}(X)$ denotes the set of $\mu \in \M_G(X)$ such that $G \acts (X, \mu)$ is aperiodic.

\begin{cor} \label{cor:mups}
Let $G$ be a countable group, let $L$ be a totally disconnected Polish space, let $L^G$ have the product topology, and equip $\M_G(L^G)$ with the weak$^*$-topology. Assume that $\cF$ is a $G$-invariant sub-$\sigma$-algebra which is generated by a collection of clopen sets.
\begin{enumerate}
\item[\rm (i)] If $\xi$ is a finite clopen partition then the map $\mu \in \M_G^{\text{aper}}(L^G) \cup \E_G(L^G) \mapsto \rh_{G,\mu}(\xi \given \cF)$ is upper-semicontinuous. If $G$ is finitely generated then this map is upper-semicontinuous on all of $\M_G(L^G)$.
\item[\rm (ii)] If $L$ is finite then the map $\mu \in \M_G^{\text{aper}}(L^G) \cup \E_G(L^G) \mapsto \rh_G(L^G, \mu \given \cF)$ is upper-semicontinuous. If $G$ is finitely generated then this map is upper-semicontinuous on all of $\M_G(L^G)$.
\end{enumerate}
\end{cor}

\begin{proof}
Let $\cL = \{R_\ell : \ell \in L\}$ be the canonical generating partition for $L^G$, where $R_\ell = \{x \in L^G : x(1_G) = \ell\}$. Clearly $\cL$ is generating and thus $\rh_G(L^G, \mu \given \cF) = \rh_{G,\mu}(\cL \given \cF)$ for all $\mu \in \M_G(L^G)$. Since in case (ii) $\cL$ is finite and clopen, (ii) is a consequence of (i). So we prove (i). Fix a finite clopen partition $\xi$, fix a measure $\mu \in \M_G(L^G)$, and fix $\epsilon > 0$. If $G$ is not finitely generated, we require $\mu$ to be in $\M_G^{\text{aper}}(L^G) \cup \E_G(L^G)$.

Since $L$ is totally disconnected and Polish, $L$ embeds into an inverse limit, $L \subset \varprojlim L_n$, of finite spaces $L_n$, $n \in \N$. Let $\pi_n : L \rightarrow L_n$ be the corresponding quotient map. By applying $\pi_n$ coordinate-wise, we also view $\pi_n$ as a $G$-equivariant map from $L^G$ to $L_n^G$, and we set $\mu_n = (\pi_n)_*(\mu)$. Notice that $\pi_n : L^G \rightarrow L_n^G$ is continuous. Let $\alpha_n = \{R_\ell : \ell \in L_n\}$ be the clopen partition of $L^G$ where $R_\ell = \{x \in L^G : \pi_n(x(1_G)) = \ell\}$, and set $\mathcal{A} = \{\alpha_n^T : n \in \N, \ T \subseteq G \text{ finite}\}$. Notice that $\mathcal{A}$ is c-dense in $\HPrt(\mu)$. By our assumption on $\cF$ we can choose a collection $\mathcal{C}_0$ of finite clopen partitions which are c-dense in $\HPrt(\cF, \mu)$.

We next construct, for each $n \in \N$, partitions $\omega_n$ and $\chi_n$ such that $\chi_n$ is clopen, $\dR_\mu(\omega_n, \chi_n) < \epsilon/2$, $\rh_{G,\nu}(\chi_n) < \epsilon$ for all $\nu \in \M_G(L^G)$, and with $\{\omega_n : n \in \N\}$ c-dense in $\HPrt(\sfinv_G, \mu)$. When $G$ is not finitely generated we simply let $\chi_n = \omega_n$ be the trivial partition for every $n$. So now suppose that $G$ is finitely generated. Let $X_n^\infty$ denote the set of points in $L_n^G$ having infinite $G$-orbit. Notice that since $L_n$ is finite and $G$ is finitely generated, $L_n^G \setminus X_n^\infty$ is countable. We wish to build a collection of finite partitions $\{\zeta_n^k : n, k \in \N\}$ such that $\zeta_n^k$ is a partition of $L_n^G$ into $G$-invariant sets and precisely one set in $\zeta_n^k$ is infinite. We further want $\zeta_n^{k+1}$ to be finer than $\zeta_n^k$, $\zeta_{n+1}^k$ to be finer than $\pi_{n+1} \circ \pi_n^{-1}(\zeta_n^k)$, and we want $\{\zeta_n^k : k \in \N\}$ to be c-dense in $\HPrt(\sfinv_G(L_n^G), \mu_n)$. It is clear that such a collection can be constructed inductively. Now define $\omega_n = \pi_n^{-1}(\zeta_n^n)$. This diagonalization ensures that $\omega_n$ is finer than $\pi_m^{-1}(\zeta_m^k)$ whenever $\max(m, k) \leq n$, and from this it easily follows that $\{\omega_n : n \in \N\}$ is c-dense in $\HPrt(\sfinv_G, \mu)$. Now fix $n$. We will construct $\chi_n$. Each finite set $P \in \zeta_n^n$ can be approximated arbitrarily well in $\mu_n$-measure by a clopen set $U \supseteq P$. By considering clopen approximations of each finite set $P \in \zeta_n^n$ and applying Lemma \ref{lem:fcl} to each, we can obtain a clopen partition $\chi'$ satisfying $\dR_{\mu_n}(\zeta_n^n, \chi') < \epsilon / 2$ and with $\rh_{G,\nu}(\chi') < \epsilon$ for all $\nu \in \M_G(L_n^G)$. Set $\chi_n = \pi_n^{-1}(\chi')$. Then $\dR_\mu(\omega_n, \chi_n) < \epsilon / 2$. This completes the construction of the $\omega_n$'s and $\chi_n$'s. Define $\mathcal{C} = \{\omega_n \vee \gamma : n \in \N, \ \gamma \in \mathcal{C}_0\}$ and note that $\mathcal{C}$ is c-dense in $\HPrt(\sfinv_G \vee \cF, \mu)$.

By Lemma \ref{lem:outkolm} there are $\alpha \in \mathcal{A}$, $\omega_n \vee \gamma \in \mathcal{C}$, finite $T \subseteq G$, and $\beta, \chi \leq \alpha$ with
$$\sH_\mu(\beta \given \chi^T \vee \omega_n \vee \gamma) < \rh_{G,\mu}(\xi \given \cF) + \epsilon \quad \text{and} \quad \sH_\mu(\chi) + \sH_\mu(\xi \given \beta^T \vee \chi^T \vee \omega_n \vee \gamma) < \epsilon.$$
Since $\dR_\mu(\omega_n, \chi_n) < \epsilon / 2$, we have
$$\sH_\mu(\beta \given \chi^T \vee \chi_n \vee \gamma) < \rh_{G,\mu}(\xi \given \cF) + 2 \epsilon \quad \text{and} \quad \sH_\mu(\chi) + \sH_\mu(\xi \given \beta^T \vee \chi^T \vee \chi_n \vee \gamma) < 2 \epsilon.$$
Since the conditional Shannon entropy of clopen partitions is a continuous function of the measure, there is an open neighborhood $U$ of $\mu$ with
$$\sH_\nu(\beta \given \chi^T \vee \chi_n \vee \gamma) < \rh_{G, \mu}(\xi \given \cF) + 2 \epsilon \quad \text{and} \quad \sH_\nu(\chi) + \sH_\nu(\xi \given \beta^T \vee \chi^T \vee \chi_n \vee \gamma) < 2 \epsilon$$
for all $\nu \in U$. Now recall that $\rh_{G,\nu}(\chi_n) < \epsilon$ for all $\nu \in \M_G(L^G)$. Using sub-additivity, we deduce that for $\nu \in U$ the entropy $\rh_{G,\nu}(\xi \given \cF)$ is bounded by
\begin{align*}
& \rh_{G,\nu}(\chi \vee \chi_n) + \rh_{G,\nu}(\beta \given \salg_G(\chi \vee \chi_n) \vee \cF) + \rh_{G,\nu}(\xi \given \salg_G(\beta \vee \chi \vee \chi_n) \vee \cF)\\
 & \leq \rh_{G,\nu}(\chi_n) + \sH_\nu(\chi) + \sH_\nu(\beta \given \chi^T \vee \chi_n \vee \gamma) + \sH_\nu(\xi \given \beta^T \vee \chi^T \vee \chi_n \vee \gamma)\\
 & < \rh_{G,\mu}(\xi \given \cF) + 5 \epsilon.
\end{align*}
This completes the proof.
\end{proof}

Next we consider the space $\HPrt(\mu)$ of countable Borel partitions of $(X, \mu)$ having finite Shannon entropy.

\begin{cor} \label{cor:pups}
Let $G \acts (X, \mu)$ be a {\pmp} action and let $\cF$ be a $G$-invariant sub-$\sigma$-algebra. For $\alpha \in \HPrt(\mu)$, let $G \acts (Y_\alpha, \nu_\alpha)$ be the factor of $(X, \mu)$ associated to $\salg_G(\alpha) \vee \cF$, and let $\cF_\alpha$ be the image of $\cF$ in $Y_\alpha$. Let $\HPrt^{\text{aper}}(\mu)$ and $\HPrt^{\text{erg}}(\mu)$ be the set of $\alpha \in \HPrt(\mu)$ for which the action $G \acts (Y_\alpha, \nu_\alpha)$ is aperiodic or ergodic, respectively. Then the map
$$\alpha \in \HPrt^{\text{aper}}(\mu) \cup \HPrt^{\text{erg}}(\mu) \mapsto \rh_G(Y_\alpha, \nu_\alpha \given \cF_\alpha)$$
is upper-semicontinuous in the metric $\dR_\mu$. If $G$ is finitely generated then this map is upper-semicontinuous on all of $\HPrt(\mu)$.
\end{cor}

\begin{proof}
Let's assume $G$ is finitely generated; the proof for the other case will be essentially identical. Fix a countable collection of Borel sets $(D_n)_{n \in \N}$ with $\salg(\{D_n : n \in \N\}) = \cF$. Define $\theta : X \rightarrow (2^\N)^G$ by the rule $\theta(x)(g)(n) = 1$ if and only if $g^{-1} \cdot x \in D_n$. Notice that $\theta^{-1}(\Borel((2^\N)^G)) = \cF \mod \Null_\mu$ and thus $G \acts ((2^\N)^G, \theta_*(\mu))$ is the (unique up to isomorphism) factor associated with $\cF$. We will work with the larger space $\N^G \times (2^\N)^G = (\N \times 2^\N)^G$. We let $\cF'$ denote the sub-$\sigma$-algebra of $(\N \times 2^\N)^G$ consisting of sets which are measurable with respect to the second component, $(2^\N)^G$. Notice that $\cF'$ is generated by a collection of clopen sets.

For any countable partition $\alpha$ of $X$ and any injection $f : \alpha \rightarrow \N$, define $\phi_{\alpha, f} : X \rightarrow \N^G$ by the rule $\phi_{\alpha, f}(x)(g) = k$ if and only if $g^{-1} \cdot x \in A \in \alpha$ and $f(A) = k$. Combining with $\theta$, we obtain the map $\phi_{\alpha, f} \times \theta : X \rightarrow (\N \times 2^\N)^G$. Set $\mu_{\alpha, f} = (\phi_{\alpha,f} \times \theta)_*(\mu)$ and observe that $G \acts ((\N \times 2^\N)^G, \mu_{\alpha,f})$ is isomorphic to $G \acts (Y_\alpha, \nu_\alpha)$ and that this isomorphism identifies $\cF'$ with $\cF_\alpha$. Therefore
$$\rh_G((\N \times 2^\N)^G, \mu_{\alpha, f} \given \cF') = \rh_G(Y_\alpha, \nu_\alpha \given \cF_\alpha).$$

Now fix $\alpha \in \HPrt(\mu)$ and fix $\epsilon > 0$. Choose an injection $f : \alpha \rightarrow \N$ whose image has infinite complement. Let $\cL = \{R_k : k \in \N\}$ be the partition of $(\N \times 2^\N)^G$ where $R_k = \{y : y(1_G) \in \{k\} \times 2^\N\}$. Notice that $\cL$ is a clopen partition and that $\alpha = (\phi_{\alpha, f} \times \theta)^{-1}(\cL)$. Next, choose a finite partition $\beta$ coarser than $\alpha$ with $\sH_\mu(\alpha \given \beta) < \epsilon$. Let $\xi$ be the corresponding coarsening of $\cL$, specifically $\xi = \{C_B : B \in \beta\} \cup \{C_\varnothing\}$ where $C_B = \cup \{R_k : f^{-1}(k) \subseteq B\}$ and $C_\varnothing = \cup \{R_k : k \not\in f(\alpha)\}$. Then $\xi$ is a finite clopen partition, $\beta = (\phi_{\alpha,f} \times \theta)^{-1}(\xi)$, and
$$\sH_{\mu_{\alpha,f}}(\cL \given \xi) = \sH_\mu(\alpha \given \beta) < \epsilon.$$

By Corollary \ref{cor:mups}, there is a weak$^*$ open neighborhood $U$ of $\mu_{\alpha, f}$ such that for all $\nu \in U$
$$\rh_{G,\nu}(\xi \given \cF') < \rh_{G,\mu_{\alpha,f}}(\xi \given \cF') + \epsilon.$$
Since $(\N \times 2^\N)^G$ is compact and totally disconnected, there are a finite number of clopen sets $W_1, \ldots, W_m$ and $\delta > 0$ such that for all $\nu \in \M_G(L^G)$
$$\Big( \forall 1 \leq i \leq m \ |\nu(W_i) - \mu_{\alpha, f}(W_i)| < \delta \Big) \Longrightarrow \nu \in U.$$
For any countable partition $\alpha'$ of $X$ and injection $f' : \alpha' \rightarrow \N$, the pre-images of the sets $W_i$ under $\phi_{\alpha',f'} \times \theta$ will be finite intersections of $G$-translates of sets from $\alpha' \cup \{D_n : n \in \N\}$. Thus there is $\kappa > 0$ such that if $\alpha'$ and $f'$ satisfy
$$\sum_{k \in \N} \mu(f^{-1}(k) \symd f'^{-1}(k)) < \kappa$$
then $\mu_{\alpha',f'} \in U$. Finally, by \cite[Fact 1.7.7]{Do11} there is $\eta > 0$ such that if $\alpha'$ is a countable partition of $X$ with $\dR_\mu(\alpha, \alpha') < \eta$ then there is an injection $f' : \alpha' \rightarrow \N$ such that $\sum_{k \in \N} \mu(f^{-1}(k) \symd f'^{-1}(k)) < \kappa$ (here we use the fact that we chose an $f$ whose image has infinite complement, allowing $f'$ to possibly use new integers). Furthermore, we may shrink $\eta$ if necessary so that if $(\alpha', f')$ are as before and $\beta'$ is defined as $\beta' = (\phi_{\alpha',f'} \times \theta)^{-1}(\xi)$, then $\sH_\mu(\alpha' \given \beta') < 2 \epsilon$. Then, for such an $\alpha'$ and $f'$, we have
\begin{align*}
\rh_G(Y_{\alpha'}, \nu_{\alpha'} \given \cF_{\alpha'}) & = \rh_G((\N \times 2^\N)^G, \mu_{\alpha',f'} \given \cF')\\
 & \leq \rh_{G,\mu_{\alpha',f'}}(\xi \given \cF') + \rh_G((\N \times 2^\N)^G, \mu_{\alpha',f'} \given \salg_G(\xi) \vee \cF')\\
 & < \rh_{G, \mu_{\alpha,f}}(\xi \given \cF') + \epsilon + \sH_{\mu_{\alpha',f'}}(\cL \given \xi)\\
 & \leq \rh_G((\N \times 2^\N)^G, \mu_{\alpha,f} \given \cF') + \epsilon + \sH_\mu(\alpha' \given \beta')\\
 & < \rh_G((\N \times 2^\N)^G, \mu_{\alpha,f} \given \cF') + 3 \epsilon\\
 & = \rh_G(Y_\alpha, \nu_\alpha \given \cF_\alpha) + 3 \epsilon.
\end{align*}
This completes the proof in the case $G$ is finitely generated, and the proof for the non-finitely generated case is essentially identical.
\end{proof}

For our final upper-semicontinuity result we consider the space of {\pmp} $G$-actions. Specifically, let $(X, \mu)$ be a standard probability space with $\mu$ non-atomic, let $\Aut(X, \mu)$ denote the group of $\mu$-preserving Borel bijections of $X$ modulo agreement $\mu$-almost-everywhere, and let $\Act(G, X, \mu)$ be the set of group homomorphisms $a : G \rightarrow \Aut(X, \mu)$. The set $\Act(G, X, \mu)$ is called the \emph{space of {\pmp} $G$-actions}. It is a Polish space under the \emph{weak topology} \cite{K10}. This topology is generated by the sub-basic open sets of the form $\{a \in \Act(G, X, \mu) : \mu(a(g)(A) \symd B) \in U\}$ for $A, B \subseteq X$ Borel and open $U \subseteq \R$.

Below we write $\PAct(G, X, \mu)$ for the set of $a \in \Act(G, X, \mu)$ for which the action $G \acts^a (X, \mu)$ is aperiodic. Similarly we let $\EAct(G, X, \mu)$ be the set of $\mu$-ergodic actions.

\begin{cor} \label{cor:actups}
Let $(X, \mu)$ be a standard probability space with $\mu$ non-atomic, let $\cP$ be a partition with $\sH(\cP) < \infty$, and let $\Sigma$ be a sub-$\sigma$-algebra. Then the map $a \in \PAct(G, X, \mu) \cup \EAct(G, X, \mu) \mapsto \rh_{a(G),\mu}(\cP \given \salg_{a(G)}(\Sigma))$ is upper-semicontinuous. If $G$ is finitely generated then this map is upper-semicontinuous on all of $\Act(G, X, \mu)$.
\end{cor}

\begin{proof}
Since all standard non-atomic probability spaces are isomorphic, we can assume without loss of generality that $X = 2^\N$. Fix a countable collection of Borel sets $\{D_n : n \in \N\}$ with $\Sigma = \salg(\{D_n : n \in \N\})$. Also fix an enumeration $\cP = \{P_n : n \in \N\}$. Set $Y = \N \times 2^\N \times 2^\N$ and define $\theta : X \rightarrow Y$ by setting $\theta(x) = (k, x, z)$ if $x \in P_k$ and for all $n \in \N$ $z(n) = 1$ precisely when $x \in D_n$. Clearly $\theta$ is a Borel injection, the pre-image of the Borel $\sigma$-algebra coming from the third component, $2^\N$, coincides with $\Sigma$, and the pre-image of the countable partition given by the first component, $\N$, coincides with $\cP$.

Consider the totally disconnected space $Y^G = \N^G \times (2^\N)^G \times (2^\N)^G$ together with the natural left-shift action of $G$. Let $\cF$ denote the $G$-invariant $\sigma$-algebra coming from the third component, $(2^\N)^G$, and define the partition $\cL = \{R_k : k \in \N\}$ by $R_k = \{y \in Y^G : y(1_G) \in \{k\} \times (2^\N)^G \times (2^\N)^G\}$. Note that $\cL$ is a clopen partition and that $\cF$ is generated by a collection of clopen sets.

For an action $a \in \Act(G, X, \mu)$, define $\phi_a : X \rightarrow Y^G$ by the rule $\phi_a(x)(g) = \theta(a(g)^{-1}(x))$. Clearly $\phi_a$ is injective (in fact $\phi_a(x)(1_G) = \theta(x)$), and $\phi_a$ is $G$-equivariant with respect to the $a$-action of $G$ on $X$. Therefore, setting $\mu_a = (\phi_a)_*(\mu)$, we have that $G \acts^a (X, \mu)$ is isomorphic to $G \acts (Y^G, \mu_a)$. Furthermore, this isomorphism identifies $\cL$ with $\cP$ and $\cF$ with $\salg_{a(G)}(\Sigma)$.

Now fix $\epsilon > 0$. Choose a finite partition $\cQ$ coarser than $\cP$ with $\sH_\mu(\cP \given \cQ) < \epsilon$. Let $\xi$ be the corresponding coarsening of $\cL$, specifically $\xi = \{C_Q : Q \in \cQ\}$ where $C_Q = \cup \{R_k : P_k \subseteq Q\}$. Now fix an action $a \in \Act(G, X, \mu)$. By Corollary \ref{cor:mups} there is a weak$^*$ open neighborhood $U$ of $\mu_a$ such that for all $\nu \in U$
$$\rh_{G,\nu}(\xi \given \cF) < \rh_{G, \mu_a}(\xi \given \cF) + \epsilon.$$
It is not difficult to check that the map $b \in \Act(G, X, \mu) \rightarrow \mu_b$ is continuous. Thus there is an open neighborhood $V$ of $a$ with $\mu_b \in U$ for all $b \in V$. Then we have
\begin{align*}
\rh_{b(G), \mu}(\cP \given \salg_{b(G)}(\Sigma)) & = \rh_{G, \mu_b}(\cL \given \cF)\\
 & \leq \rh_{G, \mu_b}(\xi \given \cF) + \rh_{G, \mu_b}(\cL \given \salg_G(\xi) \vee \cF)\\
 & < \rh_{G, \mu_a}(\xi \given \cF) + \epsilon + \sH_{\mu_b}(\cL \given \xi)\\
 & \leq \rh_{G, \mu_a}(\cL \given \cF) + \epsilon + \sH_\mu(\cP \given \cQ)\\
 & < \rh_{G, \mu_a}(\cL \given \cF) + 2 \epsilon\\
 & = \rh_{a(G), \mu}(\cP \given \salg_{a(G)}(\Sigma)) + 2 \epsilon.
\end{align*}
This completes the proof when $G$ is finitely generated. The proof for the non-finitely generated case is essentially identical.
\end{proof}

\section{Inverse limits} \label{sec:inv}

In this section we obtain a formula for the Rokhlin entropy of an inverse limit of actions. This formula was developed for ergodic actions in Part II \cite{S14b} and was a critical ingredient for the proof of the main theorem there. We believe the formula is of independent interest and will be useful for other purposes. Here we will also apply it to establish Borel measurability of Rokhlin entropy on the space of invariant measures and on the space of actions.

\begin{lem} \label{lem:invout}
Let $G \acts (X, \mu)$ be a {\pmp} action. Suppose that $G \acts (X, \mu)$ is the inverse limit of actions $G \acts (X_n, \mu_n)$. Identify each $\Borel(X_n)$ as a sub-$\sigma$-algebra of $X$ in the natural way. Let $(\cF_n)_{n \in \N}$ be an increasing sequence of sub-$\sigma$-algebras with $\cF_n \subseteq \Borel(X_n)$ for every $n$, and set $\cF = \bigvee_{n \in \N} \cF_n$. If $\cP$ is a partition with $\cP \subseteq \Borel(X_n)$ for all $n$ and $\inf_{n \in \N} \sH(\cP \given \cF_n \vee \sinv_G) < \infty$ then
$$\rh_{G,\mu}(\cP \given \cF) = \inf_{n \in \N} \rh_{G,\mu_n}(\cP \given \cF_n).$$
\end{lem}

\begin{proof}
Without loss of generality, we can assume that each $\cF_n$ is countably generated. Let $\mu = \int_{\E_G(X)} \nu \ d \tau(\nu)$ be the ergodic decomposition of $\mu$. Then every ergodic measure $\nu$ pushes forward to an ergodic measure $\nu_n$ for $G \acts X_n$, and we have $\mu_n = \int \nu_n \ d \tau(\nu)$. Pick $k \in \N$ with $\sH(\cP \given \cF_k \vee \sinv_G) < \infty$. Lemma \ref{lem:break} implies that
$$\sH(\cP \given \cF_k \vee \sinv_G) = \int_{\E_G(X)} \sH_\nu(\cP \given \cF_k) \ d \tau(\nu).$$
So for $\tau$-almost-every $\nu \in \E_G(X)$ the infimum $\inf_n \sH_\nu(\cP \given \cF_n) \leq \sH_\nu(\cP \given \cF_k)$ is finite. In \cite[Lem. 7.1]{S14b} this lemma is proven for ergodic actions. So $\rh_{G,\nu}(\cP \given \cF) = \inf_n \rh_{G,\nu_n}(\cP \given \cF_n)$ for $\tau$-almost-every $\nu \in \E_G(X)$. The claim now follows from the ergodic decomposition formula (Corollary \ref{cor:ergavg}) and the monotone convergence theorem for integrals.
\end{proof}

\begin{cor} \label{cor:relup}
Let $G \acts (X, \mu)$ be a {\pmp} action, and let $(\cF_n)_{n \in \N}$ be an increasing sequence of sub-$\sigma$-algebras. Set $\cF = \bigvee_{n \in \N} \cF_n$.
\begin{enumerate}
\item[\rm (i)] $\rh_{G,\mu}(\cP \given \cF) = \inf_{n \in \N} \rh_{G,\mu}(\cP \given \cF_n)$ if $\cP$ is a partition with $\inf_n \sH(\cP \given \cF_n \vee \sinv_G) < \infty$.
\item[\rm (ii)] $\rh_G(X, \mu \given \cF) = \inf_{n \in \N} \rh_G(X, \mu \given \cF_n)$ if the right-hand side is finite.
\end{enumerate}
\end{cor}

\begin{proof}
Clause (i) follows from Lemma \ref{lem:invout} by using $X_n = X$ for all $n$. For (ii), assume the right-hand side is finite and pick $k \in \N$ with $\rh_G(X, \mu \given \cF_k) < \infty$. Then there is a partition $\cP$ with $\sH(\cP \given \cF_k \vee \sinv_G) < \infty$ and $\salg_G(\cP) \vee \cF_k \vee \sinv_G = \Borel(X)$. So $\rh_G(X, \mu \given \cF) = \rh_{G,\mu}(\cP \given \cF)$ and $\rh_G(X, \mu \given \cF_n) = \rh_{G,\mu}(\cP \given \cF_n)$ for all $n \geq k$. By applying (i) we obtain
\begin{equation*}
\rh_G(X, \mu \given \cF) = \rh_{G,\mu}(\cP \given \cF) = \inf_{n \geq k} \rh_{G,\mu}(\cP \given \cF_n) = \inf_{n \geq k} \rh_G(X, \mu \given \cF_n).\qedhere
\end{equation*}
\end{proof}

Now we present a general formula for the Rokhlin entropy of an inverse limit.

\begin{thm} \label{thm:ks}
Let $G \acts (X, \mu)$ be a {\pmp} action and let $\cF$ be a $G$-invariant sub-$\sigma$-algebra. Suppose that $G \acts (X, \mu)$ is the inverse limit of actions $G \acts (X_n, \mu_n)$. Identify each $\Borel(X_n)$ as a sub-$\sigma$-algebra of $X$ in the natural way. Then
\begin{equation} \label{eqn:ks0}
\rh_G(X, \mu \given \cF) < \infty \Longleftrightarrow
\left\{\begin{array}{c}
\displaystyle{\inf_{n \in \N} \sup_{m \geq n} \rh_{G,\mu}(\Borel(X_m) \given \Borel(X_n) \vee \cF) = 0}\\
\displaystyle{\text{and} \quad \forall m \ \rh_{G,\mu}(\Borel(X_m) \given \cF) < \infty.}
\end{array}\right\}\end{equation}
Furthermore, when $\rh_G(X, \mu \given \cF) < \infty$ we have
\begin{equation} \label{eqn:ks}
\rh_G(X, \mu \given \cF) = \sup_{m \in \N} \rh_{G,\mu}(\Borel(X_m) \given \cF).
\end{equation}
\end{thm}

\begin{proof}
First suppose that $\rh_G(X, \mu \given \cF) < \infty$. Then
$$\rh_{G,\mu}(\Borel(X_m) \given \cF) \leq \rh_G(X, \mu \given \cF) < \infty$$
for all $m \in \N$ and by applying Corollary \ref{cor:relup}.(ii) we get
\begin{align*}
0 = \rh_G(X, \mu \given \Borel(X)) & = \inf_{n \in \N} \rh_G(X, \mu \given \Borel(X_n) \vee \cF)\\
 & \geq \inf_{n \in \N} \sup_{m \geq n} \rh_{G,\mu}(\Borel(X_m) \given \Borel(X_n) \vee \cF) \geq 0.
\end{align*}
This proves one implication in the first claim.

Now suppose that the right-side of (\ref{eqn:ks0}) is true. Fix $\delta > 0$ and for each $i \geq 1$ fix $n(i)$ with
$$\sup_{m \in \N} \rh_{G,\mu}(\Borel(X_m) \given \Borel(X_{n(i)}) \vee \cF) < \frac{\delta}{2^i}.$$
Then by using $m = n(i+1)$ we have
$$\rh_{G,\mu}(\Borel(X_{n(i+1)}) \given \Borel(X_{n(i)}) \vee \cF) < \frac{\delta}{2^i}.$$
Now by sub-additivity (Corollary \ref{cor:add2}) we have
\begin{align*}
\rh_G(X, \mu \given \cF) & \leq \rh_{G,\mu}(\Borel(X_{n(1)}) \given \cF) + \sum_{i=1}^\infty \rh_{G,\mu}(\Borel(X_{n(i+1)}) \given \Borel(X_{n(i)}) \vee \cF)\\
 & < \rh_{G,\mu}(\Borel(X_{n(1)}) \given \cF) + \delta.
\end{align*}
So $\rh_G(X, \mu \given \cF) < \infty$, completing the proof of the first claim. The second claim also follows, since above we only assumed that the right-side of (\ref{eqn:ks0}) was true (equivalently $\rh_G(X, \mu \given \cF) < \infty$ by the first claim). By letting $\delta$ tend to $0$ above, we get that $\rh_G(X, \mu \given \cF) \leq \sup_m \rh_{G,\mu}(\Borel(X_m) \given \cF)$. The reverse inequality is immediate from the definitions.
\end{proof}

It is an interesting open problem to determine if, under the assumptions of the previous theorem, one always has $\rh_G(X, \mu \given \cF) = \sup_{m \in \N} \rh_{G,\mu}(\Borel(X_m) \given \cF)$.

The formula in the previous theorem relies upon computing outer Rokhlin entropies within the largest space $X$. However, it may be more natural to use a formula which only relies upon computations occurring within the actions which build the inverse limit. With an additional assumption we can obtain such a formula.

\begin{cor} \label{cor:ks}
Let $G \acts (X, \mu)$ be a {\pmp} action. Suppose that $G \acts (X, \mu)$ is the inverse limit of actions $G \acts (X_n, \mu_n)$. Identify each $\Borel(X_n)$ as a sub-$\sigma$-algebra of $X$ in the natural way. Let $(\cF_n)_{n \in \N}$ be an increasing sequence of sub-$\sigma$-algebras with $\cF_n \subseteq \Borel(X_n)$ for every $n$, and set $\cF = \bigvee_{n \in \N} \cF_n$. Assume that $\rh_G(X_n, \mu_n \given \cF_n) < \infty$ for all $n$. Then
\begin{equation*}
\rh_G(X, \mu \given \cF) < \infty \Longleftrightarrow \inf_{n \in \N} \sup_{m \geq n} \inf_{k \geq m} \rh_{G,\mu_k}(\Borel(X_m) \given \Borel(X_n) \vee \cF_k) = 0.
\end{equation*}
Furthermore, when $\rh_G(X, \mu \given \cF) < \infty$ we have
\begin{equation*}
\rh_G(X, \mu \given \cF) = \sup_{m \in \N} \inf_{k \geq m} \rh_{G,\mu_k}(\Borel(X_m) \given \cF_k).
\end{equation*}
\end{cor}

\begin{proof}
For each $m$ pick a partition $\alpha_m \subseteq \Borel(X_m)$ with $\sH(\alpha_m \given \sinv_G(X_m) \vee \cF_m) < \infty$ and $\Borel(X_m) = \salg_G(\alpha_m) \vee \sinv_G(X_m) \vee \cF_m$. Then by Lemma \ref{lem:invout} we have
$$\rh_{G,\mu}(\Borel(X_m) \given \cF) = \rh_{G,\mu}(\alpha_m \given \cF) = \inf_{k \geq m} \rh_{G,\mu_k}(\alpha_m \given \cF_k) = \inf_{k \geq m} \rh_{G,\mu_k}(\Borel(X_m) \given \cF_k)$$
and similarly by the same reasoning for every $n \leq m$
$$\rh_{G,\mu}(\Borel(X_m) \given \Borel(X_n) \vee \cF) = \inf_{k \geq m} \rh_{G,\mu_k}(\Borel(X_m) \given \Borel(X_n) \vee \cF_k).$$
So the corollary follows from the two identities above and Theorem \ref{thm:ks}.
\end{proof}

In the next two corollaries we apply the formula in Theorem \ref{thm:ks} in order to establish the Borel measurability of Rokhlin entropy. We first consider the space of $G$-invariant probability measures.

\begin{cor} \label{cor:mborel}
Let $G$ be a countable group, let $X$ be a standard Borel space, let $G \acts X$ be a Borel action, and let $\cF$ be a countably generated $G$-invariant sub-$\sigma$-algebra.
\begin{enumerate}
\item[\rm (i)] The map $\mu \in \M_G(X) \mapsto \rh_G(X, \mu \given \cF)$ is Borel.
\item[\rm (ii)] If $\cP$ is a countable Borel partition then the map $\mu \in \{\nu \in \M_G(X) : \sH_\nu(\cP) < \infty\} \rightarrow \rh_{G,\mu}(\cP \given \cF)$ is Borel.
\end{enumerate}
\end{cor}

\begin{proof}
We first claim that $\sfinv_G$ is countably generated. Let $B$ be a Borel set which meets every finite $G$-orbit precisely once and does not meet any infinite $G$-orbit. If $Z \subseteq B$ is Borel then $G \cdot Z = \{x \in X : \exists g \in G \ g \cdot x \in Z\}$ is Borel as well. Thus $\Borel(X) \res B = \sfinv_G \res B$. Since $\sfinv_G \res B$ is isomorphic as a $\sigma$-algebra to $\sfinv_G$, and since $\Borel(X)$ is countably generated, it follows that $\sfinv_G$ is countably generated as claimed.

By the above claim and our assumption on $\cF$, $\cF \vee \sfinv_G$ is countably generated. Hence there is a countable collection $\mathcal{C}$ of finite $\cF \vee \sfinv_G$-measurable partitions which is c-dense in $\sH(\cF \vee \sfinv_G, \mu)$ for all $\mu \in \M_G(X)$. We can also fix a countable collection $\mathcal{A}$ of finite Borel partitions which is c-dense in $\HPrt(\mu)$ for all $\mu \in \M_G(X)$.

(ii). For Borel sets $D$ the map $\mu \mapsto \mu(D)$ is Borel, and similarly $\mu \mapsto \sH_\mu(\xi \given \zeta)$ is Borel for any countable Borel partitions $\xi$ and $\zeta$. So Lemma \ref{lem:outkolm} immediately implies that the map $\mu \in \{\nu \in \M_G(X) : \sH_\nu(\cP) < \infty\} \rightarrow \rh_{G,\mu}(\cP \given \cF)$ is Borel.

(i). Fix an increasing sequence of finite partitions $\alpha_n$ with $\bigvee_{n \in \N} \salg(\alpha_n) = \Borel(X)$. For each $n \in \N$ let $G \acts (X_n, \mu_n)$ be the factor of $(X, \mu)$ associated to $\salg_G(\alpha_n) \vee \cF$. Since each $\alpha_n$ is finite, it follows from (ii) that for all $n \leq m$ the functions
\begin{align*}
\mu & \mapsto \rh_{G,\mu}(\alpha_m \given \Borel(X_n) \vee \cF) = \rh_{G,\mu}(\Borel(X_m) \given \Borel(X_n) \vee \cF)\\
\text{and} \quad \mu & \mapsto \rh_{G,\mu}(\alpha_m \given \cF) = \rh_{G,\mu}(\Borel(X_m) \given \cF)
\end{align*}
are Borel. Now by applying Theorem \ref{thm:ks} we conclude that $\mu \mapsto \rh_G(X, \mu \given \cF)$ is Borel.
\end{proof}

Finally, we show that Rokhlin entropy is a Borel function on the space of actions.

\begin{cor} \label{cor:aborel}
Let $(X, \mu)$ be a standard probability space with $\mu$ non-atomic, and let $\Sigma$ be a sub-$\sigma$-algebra.
\begin{enumerate}
\item[\rm (i)] The map $a \in \Act(G, X, \mu) \rightarrow \rh_{a(G)}(X, \mu \given \salg_{a(G)}(\Sigma))$ is Borel.
\item[\rm (ii)] If $\cP$ is a countable partition with $\sH(\cP) < \infty$ then the map $a \in \Act(G, X, \mu) \rightarrow \rh_{a(G),\mu}(\cP \given \salg_{a(G)}(\Sigma))$ is Borel.
\end{enumerate}
\end{cor}

\begin{proof}
Set $Y = X^G$ and let $G$ act on $Y$ by left-shifts: $(g \cdot y)(t) = y(g^{-1} t)$ for $y \in Y$ and $g, t \in G$. Let $\pi : Y \rightarrow X$ be the map $y \mapsto y(1_G)$. Set $\bar{\cP} = \pi^{-1}(\cP)$. Let $\Sigma'$ be a countably generated $\sigma$-algebra with $\Sigma' = \Sigma \mod \Null_\mu$, and set $\bar{\Sigma} = \pi^{-1}(\Sigma')$.

For $a \in \Act(G, X, \mu)$ define $\theta^a : X \rightarrow Y = X^G$ by $\theta^a(x)(g) = a(g^{-1})(x)$. Then $\theta^a$ is a $G$-equivariant Borel injection. Set $\mu_a = \theta^a_*(\mu)$. Since $\Borel(Y)$ is generated by sets of the form $\{y \in Y : \forall t \in T \ y(t) \in B_t\}$ for finite $T \subseteq G$ and Borel sets $B_t \subseteq X$, and since
$$\mu_a(\{y \in Y : \forall t \in T \ y(t) \in B_t\}) = \mu \left( \bigcap_{t \in T} a(t)(B_t) \right),$$
we see that the map $a \in \Act(G, X, \mu) \rightarrow \mu_a \in \M_G(Y)$ is Borel.

Each map $\theta^a : (X, \mu) \rightarrow (Y, \mu_a)$ is a $G$-equivariant isomorphism with $\bar{\cP} = \theta^a(\cP) \mod \Null_{\mu_a}$ and $\bar{\Sigma} = \theta^a(\Sigma) \mod \Null_{\mu_a}$. So
$$\rh_{a(G)}(X, \mu \given \salg_{a(G)}(\Sigma)) = \rh_G(Y, \mu_a \given \salg_G(\bar{\Sigma}))$$
$$\rh_{a(G),\mu}(\cP \given \salg_{a(G)}(\Sigma)) = \rh_{G,\mu_a}(\bar{\cP} \given \salg_G(\bar{\Sigma})).$$
Using the fact that $a \mapsto \mu_a$ is Borel, and noting that $\sH_{\mu_a}(\bar{\cP}) = \sH_\mu(\cP) < \infty$, applying Corollary \ref{cor:mborel} completes the proof.
\end{proof}

\section{Comparison with Kolmogorov--Sinai and sofic entropies} \label{sec:comp}

In this section we relate Rokhlin entropy to classical Kolmogorov--Sinai entropy and sofic entropy.

As expected, we find that Rokhlin entropy and Kolmogorov--Sinai entropy coincide for free actions of amenable groups. When $\mu$ is ergodic and $\cF = \{\varnothing, X\}$, this was proven for $G = \Z$ by Rokhlin \cite{Roh67} and proven for general amenable groups by Seward and Tucker-Drob in \cite{ST14}. When $\mu$ is ergodic but $\cF$ is possibly non-trivial, this was proven by Seward in \cite{S14}.

\begin{cor} \label{cor:ksrok1}
Let $G$ be a countably infinite amenable group, let $G \acts (X, \mu)$ be a free {\pmp} action, and let $\cF$ be a $G$-invariant sub-$\sigma$-algebra. Then the relative Rokhlin and relative Kolmogorov--Sinai entropies coincide:
$$\rh_G(X, \mu \given \cF) = \ksh_G(X, \mu \given \cF).$$
In particular, $\rh_G(X, \mu) = \ksh_G(X, \mu)$.
\end{cor}

\begin{proof}
This is immediate from equality in the ergodic case \cite{S14} and the ergodic decomposition formula (Corollary \ref{cor:ergavg}).
\end{proof}

We also present a refined version of the previous corollary. For this we remind the reader the definition of Kolmogorov--Sinai entropy. Let $G$ be a countably infinite amenable group, and let $G \acts (X, \mu)$ be a free {\pmp} action. For a partition $\alpha$ and a finite set $T \subseteq G$, we write $\alpha^T$ for the join $\bigvee_{t \in T} t \cdot \alpha$, where $t \cdot \alpha = \{t \cdot A : A \in \alpha\}$. Given a $G$-invariant sub-$\sigma$-algebra $\cF$, the relative Kolmogorov--Sinai entropy is defined as
$$\ksh_G(X, \mu \given \cF) = \sup_\alpha \inf_{T \subseteq G} \frac{1}{|T|} \cdot \sH(\alpha^T \given \cF),$$
where $\alpha$ ranges over all finite Borel partitions and $T$ ranges over finite subsets of $G$ \cite{DP02}. Equivalently, one can replace the infimum with a limit over a F{\o}lner sequence $(T_n)$ \cite{OW87}. For $\xi \subseteq \Borel(X)$ we also define
$$\ksh_G(\xi \given \cF) = \sup_\alpha \inf_{T \subseteq G} \frac{1}{|T|} \cdot \sH(\alpha^T \given \cF),$$
where $\alpha$ ranges over all finite partitions which are measurable with respect to the algebra generated by $\xi$, and $T$ ranges over all finite subsets of $G$. The proof of \cite[Theorem 2.7.(i)]{DP02} can be modified to show that if $G \acts (Y, \nu)$ is the factor of $(X, \mu)$ associated to $\salg_G(\xi) \vee \cF$, then $\ksh_G(\xi \given \cF) = \ksh_G(Y, \nu \given \cF)$, where we view $\cF$ as a sub-$\sigma$-algebra of $Y$ in the natural way. In particular $\ksh_G(\xi \given \cF) = \ksh_G(\salg_G(\xi) \given \cF)$.

\begin{cor} \label{cor:ksrok2}
Let $G$ be a countably infinite amenable group, let $G \acts (X, \mu)$ be a free {\pmp} action, let $\xi \subseteq \Borel(X)$, and let $\cF$ be a $G$-invariant sub-$\sigma$-algebra. Then
$$\rh_{G,\mu}(\xi \given \cF) = \ksh_G(\xi \given \cF).$$
\end{cor}

\begin{proof}
Since both quantities satisfy an ergodic decomposition formula, it suffices to prove this with the assumption that $\mu$ is ergodic, in which case $\sinv_G$ is trivial. Fix $\epsilon > 0$. By \cite{ST14} there is a partition $\gamma$ with $\sH(\gamma) < \epsilon$ and with the property that $G$ acts freely on the factor of $G \acts (X, \mu)$ associated to $\salg_G(\gamma)$. It is not difficult to deduce from the definitions that
\begin{equation} \label{eqn:ksrok1}
\ksh_G(\xi \given \cF) \leq \ksh_G(\gamma \cup \xi \given \cF) \leq \sH(\gamma) + \ksh_G(\xi \given \cF) \leq \epsilon + \ksh_G(\xi \given \cF).
\end{equation}
Similarly, by sub-additivity of Rokhlin entropy
\begin{equation} \label{eqn:ksrok2}
\rh_{G,\mu}(\xi \given \cF) \leq \rh_{G,\mu}(\gamma \cup \xi \given \cF) \leq \sH(\gamma) + \rh_{G,\mu}(\xi \given \cF) \leq \epsilon + \rh_{G,\mu}(\xi \given \cF).
\end{equation}
Letting $G \acts (Y, \nu)$ be the factor of $(X, \mu)$ associated to $\salg_G(\gamma \cup \xi) \vee \cF$, we have that $G$ acts freely on $Y$ by construction of $\gamma$ and thus Corollary \ref{cor:ksrok1} gives
\begin{equation} \label{eqn:ksrok3}
\rh_{G,\mu}(\gamma \cup \xi \given \cF) \leq \rh_G(Y, \nu \given \cF) = \ksh_G(Y, \nu \given \cF) = \ksh_G(\gamma \cup \xi \given \cF).
\end{equation}

If $\rh_{G,\mu}(\gamma \cup \xi \given \cF) = \infty$ then we are done by (\ref{eqn:ksrok1}), (\ref{eqn:ksrok2}), and (\ref{eqn:ksrok3}). So suppose $\rh_{G,\mu}(\gamma \cup \xi \given \cF) < \infty$. Fix a partition $\beta$ with $\sH(\beta \given \cF) < \rh_{G,\mu}(\gamma \cup \xi \given \cF) + \epsilon$ and $\gamma \cup \xi \subseteq \salg_G(\beta) \vee \cF$. By (\ref{eqn:ksrok3}) we have
$$\rh_{G,\mu}(\gamma \cup \xi \given \cF) \leq \ksh_G(\gamma \cup \xi \given \cF) \leq \ksh_G(\beta \given \cF) \leq \sH(\beta \given \cF) < \rh_{G,\mu}(\gamma \cup \xi \given \cF) + \epsilon.$$
Therefore $|\rh_{G,\mu}(\xi \given \cF) - \ksh_G(\xi \given \cF)| < 3 \epsilon$ by (\ref{eqn:ksrok1}) and (\ref{eqn:ksrok2}). Now let $\epsilon$ tend to $0$.
\end{proof}

Next we compare Rokhlin entropy with sofic entropy. As a convenience to the reader, we recall the definition of sofic groups.

\begin{defn}
A countable group $G$ is \emph{sofic} if there exists a sequence of maps $\sigma_n : G \rightarrow \Sym(d_n)$ (not necessarily homomorphisms) such that
\begin{enumerate}
\item[\rm (i)] $\frac{1}{d_n} \cdot |\{1 \leq i \leq d_n : \sigma_n(g) \circ \sigma_n(h)(i) = \sigma_n(g h)(i)\}| \rightarrow 1$ for all $g, h \in G$,
\item[\rm (ii)] $\frac{1}{d_n} \cdot |\{1 \leq i \leq d_n : \sigma_n(g)(i) \neq i\}| \rightarrow 1$ for all $1_G \neq g \in G$, and
\item[\rm (iii)] $d_n \rightarrow \infty$.
\end{enumerate}
Such a sequence of maps $\Sigma = (\sigma_n : G \rightarrow \Sym(d_n))_{n \in \N}$ is called a \emph{sofic approximation} to $G$.
\end{defn}

Let $G$ be a sofic group with sofic approximation $\Sigma$. For a {\pmp} action $G \acts (X, \mu)$ and $G$-invariant sub-$\sigma$-algebras $\cF_1, \cF_2$ we let $h_{\Sigma, \mu}(\cF_1 \given \cF_2 \mathbin{:} X, G)$ denote the sofic entropy of $\cF_1$ relative to $\cF_2$ in the presence of $X$ as defined in \cite[Section 2]{H2}. The following is a slight and technical improvement upon \cite[Prop. 2.12]{H2} that, by virtue of how Rokhlin entropy was defined in that paper, only applied to aperiodic actions.

\begin{prop} \label{prop:sofrok}
Let $G$ be a sofic group with sofic approximation $\Sigma$, let $G \acts (X, \mu)$ be a {\pmp} action, and let $\cF_1, \cF_2$ be $G$-invariant sub-$\sigma$-algebras of $X$. Then
$$h_{\Sigma, \mu}(\cF_1 \given \cF_2 \mathbin{:} X, G) \leq \rh_{G,\mu}(\cF_1 \given \cF_2).$$
In particular, the sofic entropy of $G \acts (X, \mu)$ is at most $\rh_G(X, \mu)$.
\end{prop}

\begin{proof}
The final claim follows by setting $\cF_1 = \Borel(X)$ and $\cF_2 = \{X, \varnothing\}$. When $G \acts (X, \mu)$ is aperiodic this is \cite[Prop. 2.12]{H2} (that paper defines Rokhlin entropy via the formula in our Corollary \ref{cor:simplerok}). Fix a probability space $(L, \lambda)$ with $0 < \sH(L, \lambda) < \infty$. Consider the free action $G \acts (X \times L^G, \mu \times \lambda^G)$. We view $\cF_1$ and $\cF_2$ as $G$-invariant sub-$\sigma$-algebras of both $\Borel(X)$ and $\Borel(X \times L^G)$ in the natural way. It was shown by Bowen in \cite{B10b} that $h_G^\Sigma(X \times L^G, \mu \times \lambda^G) = h_G^\Sigma(X, \mu) + \sH(L, \lambda)$, and his proof also easily implies
$$h_{\Sigma, \mu \times \lambda^G}(\cF_1 \given \cF_2 \mathbin{:} X \times L^G, G) = h_{\Sigma,\mu}(\cF_1 \given \cF_2 \mathbin{:} X, G).$$
The action of $G$ on $X \times L^G$ is aperiodic, so \cite[Prop. 2.12]{H2} implies that
$$h_{\Sigma, \mu \times \lambda^G}(\cF_1 \given \cF_2 \mathbin{:} X \times L^G, G) \leq \rh_{G,\mu \times \lambda^G}(\cF_1 \given \cF_2).$$
Finally, it follows immediately from the definitions that
\begin{equation*}
\rh_{G,\mu \times \lambda^G}(\cF_1 \given \cF_2) \leq \rh_{G,\mu}(\cF_1 \given \cF_2).\qedhere
\end{equation*}
\end{proof}

\section{Restricted orbit equivalence} \label{sec:oe}

Recall that two {\pmp} actions $G \acts (X, \mu)$ and $\Gamma \acts (Y, \nu)$ are \emph{orbit equivalent} if there is a measure space isomorphism $\phi : (X, \mu) \rightarrow (Y, \nu)$ which sends almost-every $G$-orbit to a $\Gamma$-orbit. In other words, up to an isomorphism $G$ and $\Gamma$ both act on $(X, \mu)$ and they have the same orbits $\mu$-almost-everywhere.

It is a theorem of Ornstein and Weiss that any two free actions of countably infinite amenable groups are orbit equivalent \cite{OW80}. Thus orbit equivalences do not respect entropy. However, in 2000 Rudolph and Weiss made the surprising discovery that Kolmogorov--Sinai entropy is preserved under a certain restricted class of orbit equivalences \cite{RW00}. In this section we will show that Rokhlin entropy is preserved under this same restricted class of orbit equivalences. We remark that due to the definition of Rokhlin entropy this is a rather simple fact, but working from the definition of Kolmogorov--Sinai entropy, as Rudolph--Weiss did, requires more work.

Recall that for a {\pmp} action $G \acts (X, \mu)$ the \emph{induced orbit equivalence relation} is
$$E_G^X = \{(x, y) \: \exists g \in G, \ \ g \cdot x = y\}.$$
Also, the \emph{full group} of $E_G^X$, denoted $[E_G^X]$, is the set of all Borel bijections $\theta : X \rightarrow X$ with $\theta(x) \in G \cdot x$ for all $x \in X$.

\begin{defn}
Let $G \acts (X, \mu)$ be a {\pmp} action, let $\theta \in [E_G^X]$, and let $\cF$ be a $G$-invariant sub-$\sigma$-algebra. We say that $\theta$ is \emph{$\cF$-expressible} if there is an $\cF$-measurable partition $\{Z_g^\theta : g \in G\}$ of $X$ such that $\theta(x) = g \cdot x$ for almost-every $x \in Z_g^\theta$ and all $g \in G$.
\end{defn}

Notice that the partition $\{Z_g^\theta : g \in G\}$ is not unique if $G$ does not act freely. The notion of expressibility can also be stated in terms of cocycles. Specifically, $\theta \in [E_G^X]$ is $\cF$-expressible if and only if there is an $\cF$-measurable cocycle $c : \Z \times X \rightarrow G$ satisfying $c(n, x) \cdot x = \theta^n(x)$ for all $n \in \Z$ and $x \in X$.

We recall two elementary lemmas from Part I \cite{S14}.

\begin{lem}\cite[Lem. 3.2]{S14} \label{lem:expmove}
Let $G \acts (X, \mu)$ be a {\pmp} action and let $\cF$ be a $G$-invariant sub-$\sigma$-algebra. If $\theta \in [E_G^X]$ is $\cF$-expressible and $A \subseteq X$, then $\theta(A)$ is $\salg_G(\{A\}) \vee \cF$-measurable. In particular, if $A \in \cF$ then $\theta(A) \in \cF$.
\end{lem}

\begin{lem}\cite[Lem. 3.3]{S14} \label{lem:expgroup}
Let $G \acts (X, \mu)$ be a {\pmp} action and let $\cF$ be a $G$-invariant sub-$\sigma$-algebra. If $\theta, \phi \in [E_G^X]$ are $\cF$-expressible then so are $\theta^{-1}$ and $\theta \circ \phi$.
\end{lem}

The following proposition was stated for ergodic actions in Part I \cite{S14}. In the case of free actions of amenable groups it recovers the entropy preservation result of Rudolph--Weiss \cite{RW00} (by Corollary \ref{cor:ksrok2}).

Note that if $G$ and $\Gamma$ act on $(X, \mu)$ with the same orbits then $E_G^X = E_\Gamma^X$ and $[E_G^X] = [E_\Gamma^X]$. In this situation, we say that $\theta \in [E_G^X]$ is $(G, \cF)$-expressible if it is $\cF$-expressible with respect to the $G$-action $G \acts (X, \mu)$.

\begin{prop} \label{prop:oe}
Let $G \acts (X, \mu)$ and $\Gamma \acts (X, \mu)$ be aperiodic {\pmp} actions having the same orbits, and let $\cF$ be a $G$ and $\Gamma$ invariant sub-$\sigma$-algebra. If $\Gamma$ is $(G, \cF)$-expressible and $G$ is $(\Gamma, \cF)$-expressible, then for every $\xi \subseteq \Borel(X)$
$$\rh_{G,\mu}(\xi \given \cF) = \rh_{\Gamma,\mu}(\xi \given \cF).$$
In particular, $\rh_G(X, \mu \given \cF) = \rh_\Gamma(X, \mu \given \cF)$.
\end{prop}

\begin{proof}
Note that $\sinv_G = \sinv_\Gamma$. Denote this common $\sigma$-algebra by $\sinv$. Since $\Gamma$ is $(G, \cF)$-expressible, for every partition $\alpha$ Lemma \ref{lem:expmove} implies $\salg_\Gamma(\alpha) \vee \sinv \vee \cF \subseteq \salg_G(\alpha) \vee \sinv \vee \cF$. Similarly, since $G$ is $(\Gamma, \cF)$-expressible we get $\salg_G(\alpha) \vee \sinv \vee \cF \subseteq \salg_\Gamma(\alpha) \vee \sinv \vee \cF$. So for every partition $\alpha$ we have $\salg_G(\alpha) \vee \sinv \vee \cF = \salg_\Gamma(\alpha) \vee \sinv \vee \cF$. The claim now follows immediately from the definition of Rokhlin entropy.
\end{proof}

Before ending this section, we briefly mention one additional observation which seems worth recording. The following lemma is a generalization of the following simple fact: if $G \acts (X, \mu)$ is a {\pmp} action, $\Gamma$ is a subgroup of $G$, and the restricted action $\Gamma \acts (X, \mu)$ is aperiodic, then $\rh_G(X, \mu) \leq \rh_\Gamma(X, \mu)$. In the lemma below, we consider not only the case where $\Gamma$ is a subgroup of $G$ but more generally the case where $\Gamma$ is an $\cF$-expressible subgroup of the full group $[E_G^X]$. This is indeed more general, as each $g \in G$, when viewed as an element of $[E_G^X]$, is $\{X, \varnothing\}$-expressible.

\begin{lem} \label{lem:sub}
Let $G \acts (X, \mu)$ be a {\pmp} action, let $\xi \subseteq \Borel(X)$, and let $\cF$ be a $G$-invariant sub-$\sigma$-algebra. If $\Gamma \leq [E_G^X]$ is an $\cF$-expressible subgroup which acts aperiodically then
$$\rh_{G,\mu}(\xi \given \cF) \leq \rh_{\Gamma, \mu}(\xi \given \cF).$$
In particular, if $\salg_G(\xi) \vee \cF \vee \sinv_G = \Borel(X)$, then $\rh_G(X, \mu \given \cF) \leq \rh_{\Gamma,\mu}(\xi \given \cF)$.
\end{lem}

\begin{proof}
Fix $\epsilon > 0$. Since the action of $\Gamma$ is aperiodic, by Corollary \ref{cor:invzero} there is a two-piece partition $\chi$ such that $\sH(\chi) < \epsilon$ and $\sinv_\Gamma \subseteq \salg_\Gamma(\chi)$. Let $\alpha$ be a countable partition satisfying $\sH(\alpha \given \cF \vee \sinv_\Gamma) \leq \rh_{\Gamma,\mu}(\xi \given \cF) + \epsilon$ and $\xi \subseteq \salg_\Gamma(\alpha) \vee \cF \vee \sinv_\Gamma$. Using Lemma \ref{lem:expmove} we obtain $\sinv_\Gamma \subseteq \salg_\Gamma(\chi) \subseteq \salg_G(\chi) \vee \cF$ and
$$\xi \subseteq \salg_\Gamma(\alpha) \vee \sinv_\Gamma \vee \cF \subseteq \salg_\Gamma(\alpha \vee \chi) \vee \cF \subseteq \salg_G(\alpha \vee \chi) \vee \cF.$$
By sub-additivity we obtain
$$\rh_{G,\mu}(\xi \given \cF) \leq \sH(\chi) + \sH(\alpha \given \salg_G(\chi) \vee \cF) \leq \epsilon + \sH(\alpha \given \sinv_\Gamma \vee \cF) \leq \rh_{\Gamma,\mu}(\xi \given \cF) + 2 \epsilon.$$
Now let $\epsilon$ tend to $0$.
\end{proof}

\section{Stabilizers} \label{sec:stab}

In this section we look at how stabilizers relate to entropy. Before the main theorem, we need a simple lemma. Below, for an equivalence relation $R$ and $B \subseteq X$ we write $[B]_R = \{x \in X : \exists b \in B \text{ with } x \ R \ b\}$ for the $R$-saturation of $B$.

\begin{lem} \label{lem:expeq}
Let $G \acts (X, \mu)$ be a {\pmp} action, let $\cF$ be a $G$-invariant sub-$\sigma$-algebra, let $\Theta \subseteq [E_G^X]$ be a countable collection of $\cF$-expressible functions, and let $R$ be the equivalence relation generated by $\Theta$ (meaning $R$ is the smallest equivalence relation satisfying $x \ R \ \theta(x)$ for all $x \in X$ and $\theta \in \Theta$). Then for $B \subseteq X$ we have $[B]_R \in \salg_G(B) \vee \cF$.
\end{lem}

\begin{proof}
By Lemma \ref{lem:expgroup}, all combinations of elements of $\Theta$ and their inverses are $\cF$-expressible. Denote by $\langle \Theta \rangle$ the countable group generated by $\Theta$. The claim follows from Lemma \ref{lem:expmove} since $[B]_R = \bigcup_{\theta \in \langle \Theta \rangle} \theta(B) \in \salg_G(B) \vee \cF$.
\end{proof}

We also need the following fairly well known fact.

\begin{lem} \label{lem:indep}
Let $G \acts (X, \mu)$ be a {\pmp} action, let $Z \subseteq X$ be a non-null Borel set, and let $T \subseteq G$ be finite. Then there is a non-null $Z' \subseteq Z$ with $T \cdot z_1 \cap T \cdot z_2 = \varnothing$ for all $z_1 \neq z_2 \in Z'$.
\end{lem}

\begin{proof}
Let $\Gamma$ be the Borel graph on $Z$ with edge set $\{(z, z') \in Z \times Z : z \neq z' \text{ and } T \cdot z \cap T \cdot z' \neq \varnothing\}$. The degree of $\Gamma$ is at most $|T|^2$, so by \cite[Prop. 4.6]{KST99} there is a Borel function $f : Z \rightarrow \{0, 1, \ldots, |T|^2\}$ with $f(z) \neq f(z')$ for all $z, z' \in Z$ joined by an edge. Now pick any $i \in \{0, 1, \ldots, |T|^2\}$ with $\mu(f^{-1}(i)) > 0$ and set $Z' = f^{-1}(i)$.
\end{proof}

\begin{thm} \label{thm:stab}
Let $G \acts (X, \mu)$ and $G \acts (Y, \nu)$ be {\pmp} actions with the action on $X$ aperiodic, and let $\cF$ be a $G$-invariant sub-$\sigma$-algebra of $Y$.Consider a factor map $f : G \acts (X, \mu) \rightarrow G \acts (Y, \nu)$. Identify $\Borel(Y)$ as a sub-$\sigma$-algebra of $X$ in the natural way via $f$.
\begin{enumerate}
\item[\rm (i)] Assume $|\Stab_G(f(x)) : \Stab_G(x)| \geq k$ for $\mu$-almost-every $x \in X$. Then $\rh_{G,\mu}(\xi \given \cF) \leq \frac{1}{k} \cdot \rh_{G,\nu}(\xi \given \cF)$ for every collection $\xi \subseteq \Borel(Y)$. In particular,
$$\rh_{G,\mu}(Y, \nu) \leq \frac{1}{k} \cdot \rh_G(Y, \nu).$$
\item[\rm (ii)] If $|\Stab_G(f(x)) : \Stab_G(x)| = \infty$ for $\mu$-almost-every $x \in X$ then
$$\rh_{G, \mu}(Y, \nu) = 0.$$
\end{enumerate}
\end{thm}

\begin{proof}
Our proof uses some ideas of Meyerovitch \cite{Me15}. We will prove this in the case that $\mu$ is ergodic, as then the general case is obtained by Corollary \ref{cor:ergavg}. Let $R$ be the equivalence relation where $x, x' \in X$ are $R$-equivalent if and only if they lie in the same $G$-orbit and have the same image under $f$. Note that $[x]_R = \Stab_G(f(x)) \cdot x$ has cardinality $|\Stab_G(f(x)) : \Stab_G(x)|$. Since $f$ is $G$-equivariant, we have that $g \cdot [x]_R = [g \cdot x]_R$ for all $g \in G$ and $x \in X$. Thus, a single $R$-class determines all other $R$-classes in the same $G$-orbit.

(i). Fix $\xi \subseteq \Borel(Y)$ and fix $\epsilon > 0$. By ergodicity and by picking a larger $k$ if necessary, we may assume that $|[x]_R| = k$ for almost-every $x \in X$ (the case $k = \infty$ is handled by case (ii)). Since there are only countably many subsets of $G$ of cardinality $k$, there is $T \subseteq G$ such that $|T| = k$, $1_G \in T$, and with the property that $Z = \{x \in X : [x]_R = T \cdot x\}$ has positive measure. Set $\zeta = \{Z, X \setminus Z\}$. Using Lemma \ref{lem:indep}, we can replace $Z$ with a non-null subset so that $\sH(\zeta) < \epsilon / 2$ and $T \cdot z \cap T \cdot z' = \varnothing$ for all $z \neq z' \in Z$. For $t \in T$ define $\theta_t \in [E_G^X]$ by setting $\theta_t(x) = t \cdot x$ for $x \in Z$, $\theta_t(x) = t^{-1} \cdot x$ for $x \in t \cdot Z$, and $\theta_t(x) = x$ in all other cases. Then by ergodicity and up to discarding a null set, $R$ coincides with the equivalence relation generated by the $\salg_G(\zeta)$-expressible maps $\{g^{-1} \theta_t g : g \in G, \ t \in T\} \subseteq [E_G^X]$.

Fix an enumeration $1_G = g_0, g_1, \ldots$ for $G$, and define $\psi : X \rightarrow T \cdot Z = [Z]_R$ by setting $\psi(x) = g_i \cdot x$, where $i$ is least with $g_i \cdot x \in T \cdot Z = [Z]_R$, or equivalently $g_i \cdot [x]_R \subseteq [Z]_R$. Set $M = \psi^{-1}(Z)$. Since $Z$ meets every $R$-class in $[Z]_R$ precisely once, we have that $M$ meets every $R$-class precisely once and hence $\mu(M) = 1 / k$. Let $\mu_M$ be the normalized restriction of $\mu$ to $M$ defined by $\mu_M(B) = \mu(M \cap B) / \mu(M)$. Note that $M \in \salg_G(\zeta)$ and that $\mu_M(B) = \mu(B)$ for every $R$-invariant Borel set $B$.

Let $\alpha$ be a partition of $Y$ with $\xi \subseteq \salg_G(\alpha) \vee \cF$ and $\sH(\alpha \given \cF) \leq \rh_{G,\nu}(\xi \given \cF) + \epsilon / 2$. We also view $\alpha \subseteq \Borel(Y) \subseteq \Borel(X)$ as a partition of $X$. Each $A \in \alpha$ is $R$-invariant and thus $A = [A \cap M]_R$. Moreover, since every set in $\alpha \cup \cF$ is $R$-invariant we have that $\mu$ and $\mu_M$ agree on $\alpha \cup \cF$ and thus
$$\mu(M) \cdot \sH_{\mu_M}(\alpha \given \cF) = \frac{1}{k} \cdot \sH_\mu (\alpha \given \cF) \leq \frac{1}{k} \cdot \rh_{G,\nu}(\xi \given \cF) + \epsilon / 2.$$
Let $\beta$ be the join of $\{X \setminus M\} \cup (\alpha \res M)$ with $\zeta$. For each $A \in \alpha$ Lemma \ref{lem:expeq} implies that
$$A = [A \cap M]_R \in \salg_G(A \cap M) \vee \salg_G(\zeta) \subseteq \salg_G(\beta),$$
and thus $\xi \subseteq \salg_G(\alpha) \vee \cF \subseteq \salg_G(\beta) \vee \cF$. So by sub-additivity
$$\rh_{G,\mu}(\xi \given \cF) \leq \sH(\zeta) + \sH(\beta \given \salg_G(\zeta) \vee \cF) \leq \epsilon / 2 + \mu(M) \cdot \sH_{\mu_M}(\alpha \given \cF) \leq \epsilon + \frac{1}{k} \cdot \rh_{G,\nu}(\xi \given \cF).$$
Since $\epsilon > 0$ was arbitrary, this complete the proof of (i).

(ii). Fix an increasing sequence of finite partitions $(\alpha_n)_{n \in \N}$ of $Y$ satisfying $\bigvee_{n \in \N} \salg_G(\alpha_n) = \Borel(Y)$. If $\rh_{G,\mu}(\alpha_n) = 0$ for each $n$, then by sub-additivity we have $\rh_{G,\mu}(Y, \nu) \leq \sum_{n \in \N} \rh_{G,\mu}(\alpha_n) = 0$. So it suffices to fix a finite partition $\alpha$ of $Y$ and show that $\rh_{G,\mu}(\alpha) = 0$.

Fix $\epsilon > 0$. By assumption $|[x]_R| = \infty$ for almost-every $x \in X$. For each $n \geq 1$, as in the proof of (i) we can pick a finite $T_n \subseteq G$ and a non-null Borel set $Z_n \subseteq X$ such that $|T_n| = n$, $|T_n \cdot z| = n$, and $T_n \cdot z \subseteq [z]_R$ for all $z \in Z_n$. By replacing $Z_n$ with a non-null subset if necessary and applying Lemma \ref{lem:indep}, we may assume that $T_n \cdot z \cap T_n \cdot z' = \varnothing$ for all $z \neq z' \in Z_n$ and that $\sH(\zeta_n) < \epsilon / 2^{n+1}$, where $\zeta_n = \{Z_n, X \setminus Z_n\}$. As before, for $t \in T_n$ define $\theta^n_t \in [E_G^X]$ by $\theta^n_t(x) = t \cdot x$ for $x \in Z_n$, $\theta^n_t(x) = t^{-1} \cdot x$ for $x \in t \cdot Z_n$, and $\theta^n_t(x) = x$ in all other cases. Each $\theta^n_t$ is $\salg_G(\zeta_n)$-expressible. Set $\Sigma = \bigvee_{n \in \N} \salg_G(\zeta_n)$ and note that by sub-additivity $\rh_{G,\mu}(\Sigma) \leq \sum_{n \in \N} \sH(\zeta_n) < \epsilon / 2$.

Let $S$ be the equivalence relation generated by the $\Sigma$-expressible maps $\{g^{-1} \theta^n_t g : g \in G, \ n \in \N, \ t \in T_n\}$. Then $S$ is a sub-relation of $R$ and by ergodicity almost-every $S$-class is infinite. Pick a Borel set $M \subseteq X$ which meets every $S$-class but has small enough measure that $\sH(\{M, X \setminus M\}) + \mu(M) \cdot \log |\alpha| < \epsilon / 2$. Again let $\mu_M$ denote the normalized restriction of $\mu$ to $M$. Set $\beta = \{X \setminus M\} \cup (\alpha \res M)$ and observe $\sH(\beta) = \sH(\{M, X \setminus M\}) + \mu(M) \cdot \sH_{\mu_M}(\alpha) < \epsilon / 2$. Since each $A \in \alpha$ is $S$-invariant and $M$ meets every $S$-class, we have
$$A = [A \cap M]_S \in \salg_G(A \cap M) \vee \Sigma \subseteq \salg_G(\beta) \vee \Sigma$$
by Lemma \ref{lem:expeq} and thus $\alpha \subseteq \salg_G(\beta) \vee \Sigma$. It follows from sub-additivity that
$$\rh_{G,\mu}(\alpha) \leq \rh_{G,\mu}(\Sigma) + \sH(\beta) < \epsilon.$$
Letting $\epsilon$ tend to $0$, we obtain $\rh_{G,\mu}(\alpha) = 0$.
\end{proof}

The previous theorem leads to an alternate proof of Meyerovitch's theorem which states that ergodic actions of positive sofic entropy must have finite stabilizers. For a sofic group $G$ with sofic approximation $\Sigma$ and a {\pmp} action $G \acts (X, \mu)$, we let $h_G^\Sigma(X, \mu)$ denote the corresponding sofic entropy (see for instance \cite{Ke13} for the definition).

\begin{cor}
Let $G$ be a sofic group with sofic approximation $\Sigma$ and let $G \acts (Y, \nu)$ be a {\pmp} action.
\begin{enumerate}
\item[\rm (i)] $h_G^\Sigma(Y, \nu) \leq \frac{1}{k} \cdot \rh_G(Y, \nu)$ if all stabilizers have cardinality at least $k$.
\item[\rm (ii)] $h_G^\Sigma(Y, \nu) = 0$ if all stabilizers are infinite.
\end{enumerate}
\end{cor}

\begin{proof}
(i). Consider the Bernoulli shift $(2^G, u_2^G)$ and set $(X, \mu) = (2^G \times Y, u_2^G \times \nu)$. Then $G \acts (X, \mu)$ is essentially free and this action factors onto $G \acts (Y, \nu)$. Therefore by the previous theorem $\rh_{G,\mu}(Y, \nu) \leq \frac{1}{k} \cdot \rh_G(Y, \nu)$. So we have
\begin{equation} \label{eqn:soficstab1}
\rh_G(X, \mu) \leq \log(2) + \rh_{G,\mu}(Y, \nu) \leq \log(2) + \frac{1}{k} \cdot \rh_G(Y, \nu).
\end{equation}
On the other hand, Bowen proved that sofic entropy is additive under direct products with Bernoulli shifts \cite{B10b}. As sofic entropy is a lower bound to Rokhlin entropy \cite[Prop. 2.12]{H2}, we obtain
\begin{equation} \label{eqn:soficstab2}
\log(2) + h_G^\Sigma(Y, \nu) = h_G^\Sigma(X, \mu) \leq \rh_G(X, \mu).
\end{equation}
Combining (\ref{eqn:soficstab1}) and (\ref{eqn:soficstab2}) completes the proof of (i).

For (ii) the argument is mostly the same, except that in place of (\ref{eqn:soficstab1}) we have the inequality $\rh_G(X, \mu) \leq \log(2) + \rh_{G,\mu}(Y, \nu) = \log(2)$.
\end{proof}

\end{document}